\documentclass[11pt, reqno]{amsart}
\usepackage{amsmath, amsthm, amsfonts, amssymb, amsbsy, bigstrut, graphicx, enumerate,  upref, longtable, comment, booktabs, array, caption, subcaption}

\usepackage{pstricks}
\usepackage{times}

\usepackage[numbers, sort&compress]{natbib}

\usepackage[breaklinks]{hyperref}

\newcommand{\I}{\mathrm{i}}

\newtheorem{thm}{Theorem}[section]
\newtheorem{lmm}[thm]{Lemma}
\newtheorem{cor}[thm]{Corollary}
\newtheorem{prop}[thm]{Proposition}

\theoremstyle{definition}

\newcommand{\ee}{\mathbb{E}}

\newcommand{\rr}{\mathbb{R}}

\newcommand{\ve}{\varepsilon}

\newcommand{\zz}{\mathbb{Z}}

\newcommand{\upi}{\underline{\pi}}

\newcommand{\sign}{\operatorname{sign}}
 
\newcommand{\fpar}[2]{\frac{\partial #1}{\partial #2}}

\newcommand{\ftw}{\mathbb{T}}

\numberwithin{equation}{section}

\usepackage[utf8]{inputenc}
\usepackage{amsmath}
\usepackage{amsfonts}
\usepackage{bbm}
\usepackage{mathtools}
\usepackage{tikz}
\usetikzlibrary{decorations.pathreplacing}
\usepackage{amsthm}
\usepackage[english]{babel}
\usepackage{fancyhdr}
\usepackage{amssymb}
\usepackage{enumitem}
\usepackage{qtree}
\usepackage{mathrsfs}

\renewcommand{\tilde}{\widetilde}

\begin{document}

\title{Universality of deterministic KPZ}
\author{Sourav Chatterjee}
\address{Departments of Mathematics and Statistics, Stanford University}
\email{souravc@stanford.edu}
\thanks{Research partially supported by NSF grant DMS-1855484}
\keywords{KPZ equation, random surface, scaling limit, universality}
\subjclass[2010]{39A12, 60H15, 35R60, 82C41}




\begin{abstract}
Consider a deterministically growing surface of any dimension, where the growth at a point is an arbitrary nonlinear function of the heights at that point and its neighboring points. Assuming that this nonlinear function is monotone, invariant under the symmetries of the lattice, equivariant under constant shifts, and twice continuously differentiable,  it is shown that any such growing surface approaches a solution of the deterministic KPZ equation in a suitable space-time scaling limit. 
\end{abstract}

\maketitle


\section{Introduction}\label{intro}
\subsection{Main result}\label{mainresult}
Let $d$ be a positive integer.  A $d$-dimensional discrete surface is a function from $\zz^d$ into $\rr$, where the value of the function at a point denotes the height of the surface at that point. We will consider discrete surfaces that evolve over time according to some deterministic local rule, to be made precise below. Let $e_1,\ldots,e_d$ be the standard basis vectors of $\rr^d$. Let $A$ denote the set $\{0, \pm e_1,\pm e_2,\ldots, \pm e_d\}$, consisting of the origin and its $2d$ nearest neighbors in $\zz^d$. Let $B := A\setminus \{0\}$. The sets $A$ and $B$ will be fixed throughout this paper.  Let $\phi:\rr^A\to \rr$ be a function. We will say that the evolution of a deterministically growing $d$-dimensional surface $f:\zz_{\ge 0} \times \zz^d \to \rr$ is driven by the function $\phi$ if for each $t\in \zz_{\ge0}$ and $x\in \zz^d$,
\begin{align}\label{kpzevolve}
f(t+1,x) = \phi((f(t,x+a))_{a\in A}). 
\end{align}
We assume that $\phi$ has the following set of properties, all of which are quite natural from a physical point of view:
\begin{itemize}
\item {\it Equivariance under constant shifts.} For $u\in \rr^A$ and $c\in \rr$, let $u+c$ denote the vector obtained by adding $c$ to each coordinate of $u$. We assume that $\phi(u+c)=\phi(u)+c$ for each $u$ and $c$. In other words, if the whole surface at time $t$ is raised by a fixed constant $c$, then the surface at time $t+1$ is also raised by the same constant amount. 
\item {\it Monotonicity.} We assume that $\phi$ is monotone increasing. That is, if $u$ dominates $v$ in each coordinate, then $\phi(u)\ge \phi(v)$. Thus, a surface that is higher than another surface everywhere at time $t$ will continue to remain so at time $t+1$.
\item {\it Invariance under lattice symmetries.} We assume that $\phi$ is invariant under the following set of symmetries. Let $u\in \rr^A$, and let $v$ be obtained by swapping $u_{e_i}$ with $u_{e_j}$ and $u_{-e_i}$ with $u_{-e_j}$ for some $i$ and $j$. Then we assume that $\phi(u)=\phi(v)$. Also, if $u$ is obtained by swapping $u_{e_i}$ with $u_{-e_i}$ for some $i$, we assume that $\phi(u)=\phi(v)$. This means that the surface does not prefer to grow differently in different lattice directions.
\item {\it Twice continuous differentiability.} We assume that $\phi$ is twice continuously differentiable. This  regularity condition is important for obtaining a KPZ scaling limit. There are examples where the previous three conditions hold, but the limit is not KPZ because $\phi$ is not $C^2$~\cite{cs21, ks88}. 
\end{itemize}

Later, we will see several examples of $\phi$ that satisfy the above conditions. For now, let us assume that we have a function $\phi$ as above. Let $g:\rr^d \to \rr$ be a Lipschitz function. For each $\ve >0$, define $g_\ve:\zz^d\to \rr$ as 
\[
g_\ve(x) := g(\ve x). 
\]
Let $f_\ve:\zz_{\ge0} \times \zz^d \to \rr$ be the function obtained using the evolution \eqref{kpzevolve} with initial condition $f_\ve(0,x)= g_\ve(x)$ for all $x$. Our goal is to obtain a scaling limit of $f_\ve$ as $\ve \to 0$. 
To obtain the scaling limit, we need to rescale time by $\ve^{-2}$ and space by $\ve^{-1}$ (that is, parabolic scaling), and also subtract off a large time-dependent `renormalization term'. For $t\in \rr$, let $[t]$ denote the greatest integer $\le t$. For a point $x = (x_1,\ldots,x_d)\in \rr^d$, let $[x]$ denote the vector $([x_1],\ldots,[x_d])$. For $t\in \rr_{\ge 0}$, $x\in \rr^d$, and $\ve>0$, define the rescaled values 
\[
t_\ve := [\ve^{-2} t], \ \ \ x_\ve := [\ve^{-1} x].
\]
Finally, define  $f^{(\ve)}: \rr_{\ge 0} \times \rr^d\to \rr$ as 
\begin{align}\label{fvedef}
f^{(\ve)}(t,x) := f_\ve(t_\ve , x_\ve) - t_\ve \phi(0).
\end{align}
To state the theorem, we need to define two quantities related to $\phi$. For $a\in A$, let $\partial_a \phi(u)$ denote the partial derivative of $\phi$ with respect to $u_a$. By the invariance under lattice symmetries, $\partial_b \phi(0)$ is the same for all $b\in B$. Let this quantity be denoted by $\beta$. Similarly, the second order derivatives $\partial_b^2 \phi(0)$ and $\partial_b \partial_{-b}\phi(0)$ do not depend on the choice $b\in B$. Define 
\[
\gamma:= \partial_b^2\phi(0) - \partial_{b}\partial_{-b} \phi(0).
\] 
The quantities $\beta$ and $\gamma$ will remain fixed throughout the rest of the paper. Note that $\beta$ is nonnegative due to the monotonicity of $\phi$. The following theorem is the main result of this paper.
\begin{thm}[Universality of deterministic KPZ]\label{kpzunivthm}
Let all notations and assumptions be as above. First, suppose that $\beta$ and $\gamma$ are both nonzero.  Then $f^{(\ve)}(t,x)$ converges pointwise on $\rr_{\ge 0}\times \rr^d$ to a function $f$ as $\ve \to0$, where $f(0,x)=g(x)$ for all $x$, and for $t>0$,
\begin{align}\label{kpzf}
f(t,x) &= \frac{\beta}{\gamma}\log \int_{\rr^d} K(t,x-y) e^{\gamma g(y)/\beta} dy,
\end{align}
where $K(t,x) = (4\pi\beta t)^{-d/2} e^{-|x|^2/4\beta t}$. The function $f$ is continuous on $\rr_{\ge0}\times\rr^d$, is infinitely differentiable on $\rr_{>0}\times \rr^d$, and solves the deterministic KPZ equation
\begin{align}\label{kpzfsol}
\partial_t f = \beta \Delta f + \gamma |\nabla f|^2
\end{align}
with initial condition $f(0,x)=g(x)$. Next, suppose that $\beta\ne 0$ and $\gamma=0$. Then $f^{(\ve)}$ converges pointwise to the limit $f$ given by $f(0,x)=g(x)$ and
\begin{align*}
f(t,x) &= \int_{\rr^d} K(t,x-y)g(y) dy
\end{align*}
for $t>0$.  The function $f$ is continuous on $\rr_{\ge0}\times\rr^d$, is infinitely differentiable on $\rr_{>0}\times \rr^d$, and solves the heat equation $\partial_t f = \beta \Delta f$ with initial condition $f(0,x)=g(x)$. Finally, if $\beta=0$, then $\gamma$ must also be zero, and in this case $f^{(\ve)}(t,x)$ converges to $g(x)$ as $\ve \to 0$ for any $t$ and $x$. 
\end{thm}
The above result is motivated by a desire to understand the universal nature of KPZ growth of random surfaces, which has been a field of intense activity in probability theory in the last ten years (see Subsection \ref{univlit} below for definitions and a quick survey). The KPZ equation in probability theory has a random component which makes it difficult to understand and study. In fact, as of now, it does not have a satisfactory definition in dimensions greater than one. On the other hand, as shown by Theorem \ref{kpzunivthm}, the deterministic KPZ equation --- a much simpler but well-defined object --- is the universal scaling limit of a large class of deterministically growing surfaces of any dimension. The hope is that this may lead to an eventual understanding of the universal nature of KPZ growth with noise in dimensions greater than one, by somehow `injecting randomness' into the proof of Theorem \ref{kpzunivthm}. This will be partially achieved in the forthcoming paper~\cite{cha21c}, where a kind of `local KPZ growth' will be proven for arbitrary scaling limits, as $\ve \to 0$, of random surfaces growing as 
\begin{align*}
f_\ve (t+1,x) = \phi(f_\ve(t,x+a))_{a\in A},\ve z_{t+1,x}),
\end{align*}
where $z_{t,x}$ are i.i.d.~bounded random variables. This family includes, for example, the familiar model of directed polymers in a random environment.

The proof of Theorem \ref{kpzunivthm} in this manuscript is a bare-hands argument based on a novel random walk representation. A simpler proof of Theorem \ref{kpzunivthm} can be given using the Crandall--Lions theory of viscosity solutions~\cite{cl84, souganidis85} and a general method of \citet{bs91}. Such a proof allows non-smooth $\phi$'s, such as the ones displayed in equations~\eqref{lppphi} and \eqref{maxmin} below, and other extensions. This has been worked out in the companion paper of~\citet{cs21}. The proof in the present paper, however, has been retained due to the independent value of the techniques developed in this paper that help in proving superconcentration of random surfaces in \cite{cha21a}, weak convergence of directed polymers to deterministic KPZ at high temperature in \cite{cha21b}, and the aforementioned local KPZ convergence in the forthcoming paper~\cite{cha21c}. 

\subsection{Examples}\label{classex}
Various classes of driving functions satisfy the four properties required for Theorem \ref{kpzunivthm}. Two special  examples where the recursion \eqref{kpzevolve} can be explicitly solved are
\[
\phi(u) = \frac{1}{2d}\sum_{b\in B} u_b  \ \ \text{ and } \ \ 
\phi(u) = \frac{1}{\theta} \log \sum_{b\in B} e^{\theta u_b},
\]
where $\theta$ is an arbitrary positive real number. The first example gives a deterministic version of  Edwards--Wilkinson growth~\cite{ew82}, and  the second one comes from directed polymers~\cite{comets17}.  

Explicit solution of \eqref{kpzevolve} seems to be impossible in almost any other situation. For instance, it is easy to see that any function of the form
\begin{align}\label{kpzform1}
\phi(u) = u_0 + \frac{1}{2d}\sum_{b\in B} q(u_b-u_0),
\end{align}
where $q$ is a twice continuously differentiable function with $0\le q'\le 1$ everywhere on $\rr$, satisfies  the required conditions. More generally, the conditions are satisfied by any function of the form
\[
\phi(u) = u_0 + c F((u_b - u_0)_{b\in B}),
\]
where $F$ is a nondecreasing symmetric $C^2$ function on $\rr^B$ with bounded derivatives, and $c$ is a sufficiently small constant, so that the derivative of $\phi$ with respect to $u_0$ is everywhere positive. These give large classes of examples that are not explicitly solvable. 

The class of examples becomes even broader if we replace the $C^2$ assumption by the assumption that $\phi$ is a Lipschitz function. For example, then it includes the function
\begin{align}\label{lppphi}
\phi(u) = \max_{a\in A} u_a
\end{align}
which generates a deterministic version of directed last-passage percolation~\cite{joh00}. 
The $C^2$ assumption can, however, be reintroduced by convolving these functions with smooth kernels. For example, if $\phi$ satisfies the equivariance, monotonicity and invariance assumptions, and is a Lipschitz function, then for any $\delta>0$, the convolution
\begin{align}\label{phidelta}
\phi_\delta(u) := \int_{\rr^A} \frac{e^{-|u-v|^2/2\delta^2}}{(2\pi \delta^2)^{d/2}} \phi(v) dv
\end{align}
also satisfies those three conditions, and moreover, is infinitely differentiable.  Thus, we can find arbitrarily close approximations to the function displayed in equation~\eqref{lppphi} that satisfy all  four conditions. It is known that without the $C^2$ assumption --- for example for the function displayed in \eqref{lppphi} ---  the scaling limit may not be the deterministic KPZ equation~\cite{cs21, ks88}. 

A natural class of examples come from Glauber dynamics of gradient Gibbs measures with convex potentials. For instance, the growing surface generated by the Glauber dynamics of the SOS model belongs to this class~\cite{caputoetal14} (see also~\cite{hhz95}).  A gradient Gibbs measure is a measure on random surfaces formally represented~as
\[
Z^{-1}\exp\biggl(-\sum_{\substack{x,y\in \zz^d,\, |x-y|=1}} V(f(x)-f(y))\biggr)\prod_{x\in \zz^d} d\lambda(f(x)),
\]
where $\lambda$ is Lebesgue measure on $\rr$,  $V:\rr\to\rr$ is an even function called the potential, and $Z$ is the (formal) normalizing constant.  The Glauber dynamics (Gibbs sampler) for generating from a such a model proceeds by regenerating $f(x)$ from the conditional distribution given $(f(x+b))_{b\in B}$ for a randomly chosen vertex $x$, and repeating ad infinitum. A natural deterministic version of this would be to update $f(x)$ as the conditional expected value of $f(x)$ given $(f(x+b))_{b\in B}$. Explicitly, this gives a growth mechanism of the type displayed in equation \eqref{kpzevolve}, with 
\begin{align}\label{phiubdef}
\phi(u) &= \frac{\int_{-\infty}^\infty t \exp\bigl(-\sum_{b\in B} V(u_b - t)\bigr)dt}{\int_{-\infty}^\infty\exp\bigl(-\sum_{b\in B} V(u_b - t)\bigr) dt}.
\end{align}
Suppose that $V$ is convex,  even, twice continuously differentiable, and sufficiently nicely behaved to allow derivatives with respect to $u$ to be moved inside the above integrals. Then $\phi$ is twice continuously differentiable. A simple calculation shows that $\phi$ equivariant under constant shifts. Invariance under lattice symmetries is obvious from the definition. To see monotonicity, note that $\phi(u)$ has no dependence on $u_0$, and for any $b\in B$, 
\begin{align*}
\fpar{\phi}{u_b} &= -\int_{-\infty}^\infty tV'(u_b-t) \rho_u(t)dt \\
&\qquad + \int_{-\infty}^\infty t\rho_u(t)dt \int_{-\infty}^\infty V'(u_b-t)\rho_u(t)dt.
\end{align*}
where $\rho_u$ is the probability density function proportional to $e^{-\sum_{a\in B} V(u_a - t)}$. Since $t\mapsto -V'(u_b-t)$ is a nondecreasing function of $t$, the FKG--Harris inequality implies that the above expression is nonnegative, and hence $\phi$ is monotone. Even though $V$ is not $C^2$ in many models of interest, one can  consider approximations by smoothing the potential. For example, in a continuous variant of the restricted solid-on-solid (RSOS) model, $V$ is given by
\[
V(x) =
\begin{cases}
0 &\text{ if } |x|\le 1,\\
\infty &\text{ if } |x|>1. 
\end{cases}
\]
With this $V$ in \eqref{phiubdef}, an easy computation shows that 
\begin{align}\label{maxmin}
\phi(u) = \frac{1}{2}(\max_{b\in B} u_b+ \min_{b\in B} u_b). 
\end{align}
This satisfies the equivariance, monotonicity and invariance conditions, but is not a $C^2$ function. As before, this can be easily remedied by convolving with a smooth kernel as in \eqref{phidelta}. 
If $V$ is not $C^2$, it is unlikely that the scaling limit is deterministic KPZ. Indeed, as will be demonstrated in the forthcoming paper~\cite{cs21}, the scaling limits of surfaces generated by the driving functions displayed in \eqref{lppphi} and \eqref{maxmin} are solutions of new, unfamiliar PDEs.


\subsection{Literature on KPZ in probability theory}\label{univlit}
The KPZ equation was introduced by \citet{kpz86} to describe the growth of a generic randomly growing surface.  Formally, the KPZ equation is given by
\begin{align}\label{kpzeq}
\partial_t f = \beta \Delta f + \gamma |\nabla f|^2 + \kappa \xi,
\end{align}
where $\xi$ is a random field known as space-time white noise, and $\beta$, $\gamma$ and $\kappa$ are real-valued  parameters. 

The validity of the KPZ heuristic has been a topic of intense investigation in the last ten years. A foundational challenge is to give a rigorous meaning to the KPZ equation. In dimension one, there are now many approaches to solving this problem, such as the Cole--Hopf solution~\cite{bg97}, regularity structures~\cite{Hai13, Hai14}, paracontrolled distributions~\cite{GP17, GIP15}, energy solutions~\cite{gj12, gj14, gjs15, gp18} and renormalization group~\cite{km17}. 

Many one-dimensional discrete processes have been shown to have the KPZ scaling limit or properties indicative of a KPZ limit, such as directed polymers in the intermediate disorder regime~\cite{akq14, acq11, dotsenko10}, polynuclear growth~\cite{ps02}, the weakly asymmetric  exclusion processes~\cite{ss10, dt16}, log-gamma polymers~\cite{bcr13} and Macdonald processes~\cite{BC14}. Many of these results are based on exact formulas derived in prior work, such as \cite{mueller91, bdj99, joh00, cosz14, ss10b, tw08a, tw08b, tw09}. Recently, more exotic objects related to the KPZ equation have been rigorously treated, such as the KPZ fixed point~\cite{mqr17}, the KPZ line ensemble~\cite{ch16}, and the Brownian landscape~\cite{dov18}. All of this is only a small sample of the enormous literature that has grown around rigorous one-dimensional KPZ in the last ten years. For surveys, see~\cite{corwin16, quastel12, qs15}. 

Beyond dimension one much less is known. Very few exact formulas are available~\cite{bf14, ct19, ps97, toninelli18}. Even the meaning of the KPZ equation is unclear in $d\ge2$. There has been some recent progress in making sense of the equation in $d\ge 2$ by regularizing the white noise term and then making either the coefficient of $|\nabla f|^2$ or the coefficient of the noise term tend to zero as the regularization is taken away~\cite{MU18, dgrz20, cd20, csz20, ccm19, ccm20, lz20, gu20}. It turns out that such limits are in fact solutions of the Edwards--Wilkinson equation~\cite{ew82}, which is the KPZ equation~\eqref{kpzeq} with $\gamma=0$. A solution of the KPZ equation with $\gamma \ne 0$ has not yet been constructed in $d\ge 2$. 

Going beyond exactly solvable models, there have been some recent efforts towards understanding the universality of the KPZ equation in dimension one. A significant progress was made by \citet{hq18}, who showed that if the $|\nabla f|^2$ term is replaced by a polynomial function of $\nabla f$, and the white noise is replaced by the mollification of itself, the $|\nabla f|^2$ term reappears in a scaling limit as the mollification is removed.  A simpler proof of the result was given by \citet{gp16}, and extensions were recently obtained by \citet{hx19} and \citet{yang20b}. In a different approach to establishing universality, non-integrable models converging to one-dimensional KPZ were exhibited by \citet{dt16} and further developed by \citet{yang20c, yang20a}.

The investigation of the deterministic KPZ equation and its variants as scaling limits of discrete growth models was initiated by \citet{ks88}. This work inspired a large body of follow-up work in physics, especially in modeling  deterministic traffic flows, such as the model of \citet{bml92}. However, not much was done on the rigorous side. The mathematical properties of the deterministic KPZ and related equations have been studied~\cite{ggk03, bsw02}, but convergence to the deterministic KPZ equation has not been rigorously established for any discrete model.




\section{Proof of the main result}
\subsection{A priori control on first-order roughness}\label{prelimsec}
Let $L$ denote the Lipschitz constant of the initial data $g$. Throughout the remainder of this paper, $C,C_0,C_1,\ldots$ will denote constants that depend only on $d$, $\phi$ and $L$. The values of these constants may change from line to line, or even within a line. We will assume without loss of generality that $\ve\in (0,1)$. 

First, let us show that there is no loss of generality in assuming that $\phi(0)=0$. This will simplify matters a bit by taking off the renormalization term in \eqref{fvedef}. Define $\tilde{\phi}(u) := \phi(u)-\phi(0)$, and let $\tilde{f}_\ve$ be obtained by evolving the surface with $\tilde{\phi}$ instead of $\phi$, but with the same initial data $g_\ve$. 
\begin{lmm}
For any $t\in \zz_{\ge 0}$ and $x\in \zz^d$, 
\[
\tilde{f}_\ve(t,x) = f_\ve(t,x)-t\phi(0).
\]
\end{lmm}
\begin{proof}
We will prove the identity by induction on $t$. It is obviously true when $t=0$, since both sides are equal to $g_\ve(x)$. Suppose that it holds for some $t$. Then by equivariance under constant shifts,
\begin{align*}
\tilde{f}_\ve(t+1,x) &= \tilde{\phi}((\tilde{f}_\ve(t,x+a))_{a\in A})= \phi((\tilde{f}_\ve(t,x+a))_{a\in A}) - \phi(0)\\
&= \phi((f_{\ve}(t,x+a))_{a\in A} - t\phi(0)) - \phi(0)\\
&= \phi((f_{\ve}(t,x+a))_{a\in A}) - (t+1)\phi(0)\\
&= f_\ve(t+1,x) - (t+1)\phi(0).
\end{align*}
This completes the proof of the lemma.
\end{proof}
The above lemma shows that we can work with $\tilde{\phi}$ and $\tilde{f}_\ve$ instead of $\phi$ and $f_\ve$. In other words, we can assume, without loss of generality, that $\phi(0)=0$. We will work under this assumption henceforth. We begin with two key observations about the function~$\phi$.
\begin{lmm}\label{derlmm}
For each $u\in \rr^A$ and $a\in A$, $\partial_a\phi(u)\ge 0$, and $\sum_a \partial_a\phi(u)=1$.
\end{lmm}
\begin{proof}
The first claim follows from monotonicity, and the second from equivariance.
\end{proof}
\begin{lmm}\label{maxlmm}
For any $u,v\in \rr^A$, $|\phi(u)-\phi(v)|\le \max_{a\in A} |u_a-v_a|$.
\end{lmm}
\begin{proof}
Note that
\begin{align*}
\phi(u)-\phi(v) &= \int_0^1 \nabla \phi(tu + (1-t)v)\cdot (u-v) dt.
\end{align*}
But by Lemma \ref{derlmm},
\begin{align*}
| \nabla \phi(tu + (1-t)v)\cdot (u-v)| &= \biggl|\sum_{a\in A} \partial_a  \phi(tu + (1-t)v) (u_a-v_a)\biggr|\\
&\le \sum_{a\in A} \partial_a  \phi(tu + (1-t)v)|u_a-v_a|\\
&\le \max_{a\in A} |u_a-v_a|.
\end{align*}
This completes the proof of the lemma.
\end{proof}
Lemma \ref{maxlmm} has the following important consequence, which is crucial for all of the subsequent estimates. It says for any $t$, $f_\ve(t,\cdot)$ is Lipschitz on $\zz^d$ under the $\ell^1$ metric with Lipschitz constant $L\ve$. In other words, the roughness of $f_\ve(t,\cdot)$ does not grow with $t$.  The monotonicity and equivariance assumptions are mostly needed for this step. Without a control such as this, it is difficult to ensure that $f_\ve$ has a smooth scaling limit as~$\ve \to 0$.
\begin{lmm}\label{lipschitzlmm}
For any $t$ and any two neighboring points $x$ and $y$, 
\[
|f_\ve(t,x)-f_\ve(t,y)|\le L\ve.
\] 
\end{lmm}
\begin{proof}
Note that
\begin{align*}
|f_\ve(0,x)-f_\ve(0,y)| &= |g_\ve(x)-g_\ve(y)|\\
&= |g(\ve x) - g(\ve y)| \le L\ve|x-y|.
\end{align*}
This proves the claim for $t=0$. Suppose that it holds for some $t$. Then by Lemma~\ref{maxlmm}, for any two neighboring points $x$ and $y$, 
\begin{align*}
|f_\ve(t+1,x)-f_\ve(t+1,y)| &= |\phi((f_\ve(t, x+a))_{a\in A}) - \phi((f_\ve(t, y+a))_{a\in A})|\\
&\le \max_{a\in A} |f_\ve(t, x+a)-f_\ve(t,y+a)|.
\end{align*}
But $x+a$ is a neighbor of $y+a$ for each $a$. Thus, $|f_\ve(t,x+a)-f_\ve(t,y+a)|\le L\ve$ for each $a$. This completes the proof.
\end{proof}


\subsection{A random walk representation of the surface}\label{rwsec}
The main result of this subsection, stated below, represents the solution of \eqref{kpzevolve} as a formula involving random walk transition probabilities. This solution is not explicit, as it will involve terms that depend on $f_\ve$. We will use it later to obtain an integral equation for the scaling limit, which will then lead to the solution displayed in Theorem \ref{kpzunivthm}.

In addition to the quantities $\beta$ and $\gamma$ that have already been defined, let us also define $\alpha := \partial_0\phi(0)$. Note that by Lemma~\ref{derlmm}, $\alpha$ and $\beta$ are nonnegative, and $\alpha + 2d\beta =1$. Consider a random walk on $\zz^d$, starting at the origin at time $0$, and jumping from $x$ to $x+a$, $a\in A$, with probability $\alpha$ if $a=0$ and $\beta$ if $a\ne 0$. Let $p(t,x)$ be the probability that this random walk is at the point $x$ at time $t$. Then note that for any $t\in \zz_{\ge1}$ and $x\in \zz^d$,
\begin{align}\label{prec}
p(t,x) &= \alpha p(t-1,x) + \beta \sum_{b\in B} p(t-1,x-b).
\end{align}
For each $x\in \zz^d$ and $t\in \zz_{\ge 1}$, let 
\begin{align}\label{htx}
h_\ve(t,x) := f_\ve(t,x) - \alpha f_\ve(t-1,x) - \beta\sum_{b\in B} f_\ve(t-1,x+b). 
\end{align}
The following result represents the function $f_\ve$ in terms of the transition probabilities of the random walk defined above and the functions $h_\ve$ and $g_\ve$. 
\begin{prop}\label{fformlmm}
For any $t\in \zz_{\ge 1}$ and $x\in \zz^d$,
\begin{align*}
f_\ve(t,x) &= \sum_{y\in \zz^d} p(t,x-y) g_\ve(y) + \sum_{0\le s\le t-1} \sum_{y\in \zz^d} p(s, x-y) h_\ve(t-s,y).
\end{align*}
\end{prop}
\begin{proof}
The proof is by induction on $t$. First, note that $p(0,0) = 1$, $p(1, 0) = \alpha$ and $p(1,b)=\beta$ for all $b\in B$. Thus, by \eqref{htx},
\begin{align*}
&\sum_{y\in \zz^d} p(1,x-y) g_\ve(y) + \sum_{y\in \zz^d} p(0, x-y) h_\ve(1,y)\\
&= \beta \sum_{b\in B} f_\ve(0, x+b) + \alpha f_\ve (0,x) + h_\ve(1,x) = f_\ve(1,x). 
\end{align*}
This proves the claim for $t=1$. Now suppose that the claim holds up to time $t-1$ for some $t\ge 2$. Then again by \eqref{htx},
\begin{align}
&f_\ve(t,x) = \alpha f_\ve(t-1,x) + \beta\sum_{b\in B} f_\ve(t-1,x+b) + h_\ve(t,x)\notag\\
&= \alpha \sum_{y\in \zz^d} p(t-1,x-y) g_\ve(y) + \alpha\sum_{0\le s\le t-2} \sum_{y\in \zz^d}p(s, x-y) h_\ve(t-1-s,y)\notag\\
&\qquad + \beta \sum_{y\in \zz^d} \sum_{b\in B} p(t-1,x+b-y) g_\ve(y)\notag\\
&\qquad + \beta \sum_{0\le s\le t-2} \sum_{y\in \zz^d} \sum_{b\in B} p(s, x+b-y) h_\ve(t-1-s,y) + h_\ve (t,x). \label{fepsilon}
\end{align}
Notice that by the relation \eqref{prec},
\begin{align*}
\alpha p(t-1,x-y) + \beta\sum_{b\in \zz^d} p(t-1, x+b-y) = p(t,x-y),
\end{align*}
and similarly,
\begin{align*}
\alpha p(s,x-y) + \beta\sum_{b\in B} p(s, x+b-y) = p(s+1,x-y).
\end{align*}
Plugging these into \eqref{fepsilon} gives
\begin{align*}
f_\ve(t,x) &= \sum_{y\in \zz^d} p(t,x-y) g_\ve(y) + \sum_{0\le s\le t-2}\sum_{y\in \zz^d} p(s+1, x-y)h_\ve(t-1-s, y) \\
&\qquad+ h_\ve(t,x)\\
&=  \sum_{y\in \zz^d} p(t,x-y) g_\ve(y) + \sum_{1\le s\le t-1}\sum_{y\in \zz^d} p(s, x-y)h_\ve(t-s, y) \\
&\qquad+ h_\ve(t,x).
\end{align*}
Since $p(0,0)=1$, completes the induction step. 
\end{proof}

\subsection{A priori control on second-order roughness}\label{basicsec}
We now prove a second-order refinement of Lemma \ref{lipschitzlmm}. This lemma shows that $|h_\ve(t,x)|$ is uniformly bounded by a constant times $\ve^2$. We will see later that this effectively gets a control on the second-order smoothness of $f_\ve$, just as Lemma \ref{lipschitzlmm} gives a control on the first-order smoothness. 
\begin{lmm}\label{lip2lmm}
For each $t\in \zz_{\ge 1}$ and $x\in \zz^d$, $|h_\ve (t,x)|\le C\ve^2$. 
\end{lmm}
\begin{proof}
Let $u\in \rr^A$ be the vector whose coordinates are $u_0=0$ and 
\[
u_b = f_\ve (t-1,x+b)-f_\ve(t-1,x)
\]
for $b\in B$. Then note that by equivariance,
\begin{align}
f_\ve(t,x) &= \phi((f_\ve(t-1,x+a))_{a\in A})\notag \\
&= f_\ve(t-1,x) + \phi(u).\label{ftx}
\end{align}
By Lemma \ref{lipschitzlmm} and the assumption that $\ve <1$, $|u_b|$ is uniformly bounded by $L\ve$ and hence by a constant that does not depend on $\ve$. Therefore, by the twice continuous differentiability of $\phi$, the assumption that $\phi(0)=0$, and Taylor expansion, 
\begin{align*}
\biggl|\phi(u) - \sum_{b\in B} u_b \partial_b \phi(0)\biggr|\le C\sum_{b\in B} u_b^2. 
\end{align*}
Since $|u_b|\le L\ve$ and $\partial_b \phi(0)=\beta$, this gives 
\begin{align}\label{phiu}
\biggl|\phi(u) - \beta \sum_{b\in B} u_b \biggr|\le C\ve^2. 
\end{align}
Finally, note that since $\alpha = 1-2d\beta$, 
\begin{align}
f_\ve(t-1,x) + \beta \sum_{b\in B} u_b &= (1-2d\beta) f_\ve(t-1,x) + \beta \sum_{b\in B} f_\ve(t-1,x+b)\notag \\
&= \alpha f_\ve(t-1,x) + \beta \sum_{b\in B} f_\ve(t-1,x+b). \label{ft1x}
\end{align}
Combining \eqref{ft1x} with the formula \eqref{htx} for $h_\ve(t,x)$, we get 
\begin{align*}
\beta\sum_{b\in B} u_b = f_\ve(t,x)-f_\ve(t-1,x)-h_\ve(t,x). 
\end{align*}
Plugging this into \eqref{phiu} gives
\begin{align*}
|f_\ve(t,x)-f_\ve(t-1,x)-h_\ve(t,x) - \phi(u)| &\le C\ve^2.
\end{align*}
But by \eqref{ftx}, $f_\ve(t,x)-f_\ve(t-1,x)- \phi(u)=0$. This completes the proof of the lemma.
\end{proof}

\subsection{Plan of the proof}\label{sketchsec}
Having stated and proved Lemma~\ref{lipschitzlmm}, Proposition~\ref{fformlmm}, and Lemma~\ref{lip2lmm}, we are now in a position to discuss the general strategy for the proof of Theorem~\ref{kpzunivthm}. The first step is to use the recursive formula for $f_\ve(t,x)$ given in Proposition~\ref{fformlmm}, together with the bound on $h_\ve$ given by Lemma \ref{lip2lmm} and a number of delicate estimates for first- and second-order discrete derivatives of random walk transition probabilities to show that $f_\ve$ has subsequential scaling limits as $\ve\to 0$, all of which are differentiable. 

The next step is to show that for any subsequential scaling limit $f$, the rescaled version of $h_\ve$ converges to $\gamma |\nabla f|^2$. This uses the formula \eqref{htx} of $h_\ve(t,x)$, together with Taylor expansion and the a priori controls given by Lemma~\ref{lipschitzlmm} and Lemma~\ref{lip2lmm}. This step uses the invariance of $\phi$ under lattice symmetries and the fact that $\phi$ is twice continuously differentiable.  Once this is established, Proposition~\ref{fformlmm} will again be used, together with a local central limit theorem for the transition probability $p(t,x)$, to show that $f$ satisfies the integral equation
\begin{align*}
f(t,x) &= \int_{\rr^d} K(t, x -y) g(y)dy \\
&\qquad \qquad +\gamma \int_0^{t} \int_{\rr^d} K(s, x-y) |\nabla f|^2(t-s, y)dy ds,
\end{align*}
where $K$ is the Gaussian kernel from the statement of Theorem \ref{kpzunivthm}. We will then show that this integral equation has a unique solution, given by the Cole--Hopf formula \eqref{kpzf} displayed in Theorem \ref{kpzunivthm}.

\subsection{Random walk transition probabilities}\label{approxsec}
In this subsection we begin to work out the necessary estimates for random walk transition probabilities that are crucial for the carrying out the plan of the proof sketched out in Subsection \ref{sketchsec}. 
We will henceforth assume that $\alpha \ne 0$. The case $\alpha = 0$ needs special treatment due to parity issues, which will be dealt with in Subsection \ref{alphazero}.

Let $\xi$ be a random vector which takes value $b$ with probability $\beta$ for each $b\in B$, and $0$ with probability $\alpha$. Then the characteristic function $\psi$ of $\xi$ is given by
\begin{align*}
\psi(\theta) = \ee(e^{\I \theta\cdot \xi}) = \alpha + 2\beta \sum_{i=1}^d \cos \theta_i
\end{align*}
for $\theta = (\theta_1,\ldots,\theta_d)\in \rr^d$. We are going to use the Fourier inversion formula for expressing the transition probability $p(t,x)$ in terms of the characteristic function $\psi$. This makes it necessary to derive two important estimates for $\psi$, as done below.
\begin{lmm}\label{psi1}
There is a positive constant $C$ depending only on $\alpha$, $\beta$,  and $d$, such that for all $\theta\in [-\pi,\pi]^d$,
$|\psi(\theta)|\le e^{-C|\theta|^2}$.
\end{lmm}
\begin{proof}
Since $\alpha>0$, $\beta>0$, and $\alpha+2\beta d = 1$, it is easy to see that the map $x\mapsto |\alpha + 2\beta d \cos x|$ equals $1$ at $x=0$ and is strictly less than $1$ everywhere else in the closed interval $[-\pi,\pi]$. From this, and the behavior of cosine near zero, it follows that the function
\begin{align*}
w(x) := 
\begin{cases}
(1- |\alpha + 2\beta d \cos x|)/x^2 &\text{ if } x\in [-\pi,\pi]\setminus \{0\},\\
\beta d  &\text{ if } x=0
\end{cases}
\end{align*}
is continuous and nonzero everywhere on $[-\pi, \pi]$. Thus, there is a positive constant $C$, depending only on $\alpha$, $\beta$ and $d$, such that $w(x)\ge C$ for all $x\in [-\pi,\pi]$. In other words, $|\alpha + 2\beta d \cos x| \le 1-Cx^2$  
for all $x\in [-\pi, \pi]$. Therefore, for any $\theta \in [-\pi,\pi]^d$, 
\begin{align*}
|\psi(\theta)| &\le \frac{1}{d}\sum_{i=1}^d |\alpha +2\beta d \cos\theta_i|\\
&\le\frac{1}{d} \sum_{i=1}^d (1-C\theta_i^2) =  1- \frac{C|\theta|^2}{d}. 
\end{align*}
The proof is completed by applying the inequality $1-x \le e^{-x}$. 
\end{proof}
\begin{lmm}\label{psi2}
For any $\theta \in \rr^d$ and $t\in \zz_{\ge 1}$, 
\[
|\psi(\theta)^t - e^{-\beta t |\theta|^2}|\le Ct|\theta|^4. 
\]
\end{lmm}
\begin{proof}
By Taylor expansion, 
\[
\biggl|\cos x - 1 + \frac{x^2}{2}\biggr| \le \frac{x^4}{24}
\]
for all $x\in \rr$. By this, and the identity $\alpha + 2\beta d=1$, we have
\begin{align*}
|\psi(\theta) - 1 +\beta |\theta|^2| &= \biggl|2\beta \sum_{i=1}^d  \biggl(\cos \theta_i - 1 + \frac{\theta_i^2}{2}\biggr)\biggr|\\
&\le  2\beta \sum_{i=1}^d \biggl|\cos \theta_i - 1 + \frac{\theta_i^2}{2}\biggr|\\
&\le \frac{\beta \sum_{i=1}^d \theta_i^4}{12} \le C|\theta|^4. 
\end{align*}
On the other hand,  again by Taylor expansion,
\begin{align*}
|e^{-x} -  1 + x| \le \frac{x^2}{2} 
\end{align*}
for all $x\ge 0$. Combining this with the preceding display, we get
\begin{align*}
|\psi(\theta)- e^{-\beta |\theta|^2}| &\le C|\theta|^4.
\end{align*}
Finally, note that the absolute values of both $\psi(\theta)$ and $e^{-\beta |\theta|^2}$ are both bounded above by $1$. Since $|a-b|^t \le t|a-b|$ for any two complex numbers $a$ and $b$ in the unit disk and any positive integer $t$, the claim follows. 
\end{proof}
Now note that by Fourier inversion, we have 
\begin{align}
p(t,x) &= (2\pi)^{-d}\int_{[-\pi, \pi]^d} e^{-\I \theta \cdot x} \psi(\theta)^t d\theta\notag\\
&= (2\pi)^{-d} t^{-d/2}\int_{[-\pi\sqrt{t}, \pi\sqrt{t}]^d} e^{-\I t^{-1/2}\eta \cdot x}(\psi(t^{-1/2}\eta ))^t d\eta,\label{finv}
\end{align}
where the second line was obtained by the change of variable $\eta = \sqrt{t}\theta$. We will now simplify this formula to get successively simpler approximations $p_1$, $p_2$ and $p_3$ for $p$, and bound the approximation errors using Lemmas \ref{psi1} and \ref{psi2}. First, let 
\[
p_1(t,x) := (2\pi)^{-d} t^{-d/2}\int_{|\eta|\le \log t} e^{-\I t^{-1/2}\eta \cdot x}(\psi(t^{-1/2}\eta ))^t d\eta.
\]
Then we have the following bound on the difference between $p$ and $p_1$.
\begin{lmm}\label{pp1lmm}
For any $t\in \zz_{\ge 1}$ and $x\in \zz^d$, $|p(t,x)-p_1(t,x)|\le C_1 e^{-C_2 (\log t)^2}$. 
\end{lmm}
\begin{proof}
Take any $\eta$ with absolute value bigger than $\log t$. Then by Lemma \ref{psi1}, 
\begin{align*}
|\psi(t^{-1/2} \eta)| &\le e^{-Ct^{-1}|\eta|^2} \le e^{-Ct^{-1}(\log t)^2}. 
\end{align*}
Thus,
\begin{align*}
|p(t,x)-p_1(t,x)|&\le C t^{-d/2}\int_{\log t < |\eta| \le \pi \sqrt{td}} |\psi(t^{-1/2} \eta)|^t d\eta \\
&\le  Ct^{-d/2}\int_{\log t < |\eta| \le \pi \sqrt{td}} e^{-C(\log t)^2} d\eta\\
&\le C_1 e^{-C_2 (\log t)^2},
\end{align*}
where in the last line we used the fact that any polynomial in $t$ is rendered irrelevant by the presence of $e^{-C_2(\log t)^2}$. This completes the proof of the lemma.
\end{proof}
Next, define
\[
p_2(t,x) :=(2\pi)^{-d} t^{-d/2}\int_{|\eta|\le \log t} e^{-\I t^{-1/2}\eta \cdot x - \beta |\eta|^2} d\eta,
\]
and finally, let 
\begin{align}
p_3(t,x) &:= (2\pi)^{-d}t^{-d/2}\int_{\rr^d} e^{-\I t^{-1/2}\eta \cdot x - \beta |\eta|^2} d\eta\notag\\
&= (4\pi \beta t)^{-d/2} e^{-|x|^2/4\beta t}.\label{p3form}
\end{align}
The following lemma shows that $p_2$ is close to $p_3$. 
\begin{lmm}\label{p2p3lmm}
For any $t\in \zz_{\ge 1}$ and $x\in \zz^d$, $|p_2(t,x)-p_3(t,x)|\le C_1 e^{-C_2(\log t)^2}$.
\end{lmm}
\begin{proof}
Using polar coordinates, we have
\begin{align*}
|p_2(t,x)-p_3(t,x)| &\le Ct^{-d/2}\int_{|\eta|> \log t} e^{-\beta |\eta|^2} d\eta\\
&= Ct^{-d/2}\int_{\log t}^\infty r^{d-1} e^{-\beta r^2} dr.
\end{align*}
It is not hard to see that the last expression is bounded by $C_1 e^{-C_2(\log t)^2}$. 
\end{proof}
Note that we have not yet established the closeness of $p_1$ and $p_2$. This will be done in the next subsection.

\subsection{Smoothness of transition probabilities}\label{transsec}
The primary goal of this subsection is to get upper bounds on the first- and second-order discrete derivatives of $p$.  The plan is to show that these can be approximated by the first- and second-order derivatives of $p_3$. The controls given by Lemma \ref{pp1lmm} and Lemma \ref{p2p3lmm} are sufficient to relate the derivatives of $p$ with those of $p_1$, and the derivatives of $p_2$ with those of $p_3$. Approximating the derivatives of $p_1$ with those of $p_2$ require more delicate control, which we will now obtain.

For any function $w:\zz^d \to \rr$, and any $x,y,\in \zz^d$, define
\[
\delta_y w(x) := w(x+y)-w(x).
\]
For any $y$ and $z$, the discrete derivative operators $\delta_y$ and $\delta_z$ commute, since
\begin{align*}
\delta_y\delta_z w(x) &= \delta_z w(x+y) - \delta_z w(x)\\
&= w(x+y+z)-w(x+y) - (w(x+z)-w(x))\\
&= w(x+y+z) - w(x+z) - (w(x+y)-w(x)) = \delta_z\delta_y w(x).
\end{align*}
The following lemma shows that $p_1$ and $p_2$ are close, as are their first- and second-order discrete derivatives. The exact orders of the error bounds are important for subsequent estimates.
\begin{lmm}\label{p1p2lmm}
For any $t\in \zz_{\ge 1}$ and $x,y,z\in \zz^d$,
\begin{align*}
|p_1(t,x)-p_2(t,x)|&\le C_1 t^{-(d+2)/2}(\log t)^{C_2},\\
|\delta_yp_1(t,x)-\delta_yp_2(t,x)|&\le C_1|y| t^{-(d+3)/2}(\log t)^{C_2},\\
|\delta_y\delta_zp_1(t,x)-\delta_y\delta_zp_2(t,x)|&\le C_1|y||z| t^{-(d+4)/2}(\log t)^{C_2}.
\end{align*}
\end{lmm}
\begin{proof}
By Lemma \ref{psi2}, 
\begin{align*}
|(\psi(t^{-1/2}\eta ))^t - e^{-\beta |\eta|^2}| &\le C t^{-1} |\eta|^4. 
\end{align*}
Thus,
\begin{align*}
|p_1(t,x)-p_2(t,x)| &\le Ct^{-d/2}\int_{|\eta|\le \log t} |(\psi(t^{-1/2}\eta ))^t - e^{-\beta |\eta|^2}|  d\eta\\
&\le Ct^{-d/2-1} \int_{|\eta|\le \log t} |\eta|^4 d\eta\\
&\le C t^{-d/2-1}(\log t)^{d+4}.
\end{align*}
This proves the first inequality in the statement of the lemma. Next, note that
\begin{align*}
\delta_yp_1(t,x) = (2\pi)^{-d}t^{-d/2}\int_{|\eta|\le \log t}  (e^{-\I t^{-1/2}\eta \cdot y} - 1)e^{-\I t^{-1/2}\eta \cdot x}(\psi(t^{-1/2}\eta ))^t d\eta
\end{align*}
and 
\begin{align*}
\delta_yp_2(t,x) &= (2\pi)^{-d}t^{-d/2}\int_{|\eta|\le \log t} (e^{-\I t^{-1/2}\eta \cdot y} - 1) e^{-\I t^{-1/2}\eta \cdot x - \beta |\eta|^2} d\eta.
\end{align*}
By the inequality $|e^{-\I a}-1|\le |a|$ for $a\in \rr$, this gives
\begin{align*}
|\delta_yp_1(t,x) - \delta_yp_2(t,x)| &\le t^{-(d+1)/2} |y| \int_{|\eta|\le \log t} |\eta| |(\psi(t^{-1/2}\eta ))^t - e^{-\beta |\eta|^2}|  d\eta.
\end{align*}
Now proceeding as before, we get the second inequality. Similarly, noting that
\begin{align*}
&\delta_y\delta_zp_1(t,x) \\
&= (2\pi)^{-d}t^{-d/2}\int_{|\eta|\le \log t}  (e^{-\I t^{-1/2}\eta \cdot y} - 1)(e^{-\I t^{-1/2}\eta \cdot z} - 1)\\
&\qquad \qquad \qquad \qquad \qquad \qquad \cdot e^{-\I t^{-1/2}\eta \cdot x}(\psi(t^{-1/2}\eta ))^t d\eta
\end{align*}
and 
\begin{align*}
&\delta_y\delta_zp_2(t,x) \\
&= (2\pi)^{-d}t^{-d/2}\int_{|\eta|\le \log t} (e^{-\I t^{-1/2}\eta \cdot y} - 1)(e^{-\I t^{-1/2}\eta \cdot z} - 1) e^{-\I t^{-1/2}\eta \cdot x - \beta |\eta|^2} d\eta,
\end{align*}
and proceeding just as above, we get the third inequality in the statement of the lemma.
\end{proof}
Now we have all the necessary ingredients for getting bounds on the derivatives of $p$ using bounds on the derivatives of $p_3$. The following lemma gives the bounds for $p_3$. 
\begin{lmm}\label{p3lmm}
Take any $t\in \zz_{\ge 1}$ and $x,y,z\in \zz^d$. Let 
\begin{align*}
a &:= \min\{|x|, |x+y|\},\\
b &:= \min\{|x|, |x+y|, |x+z|, |x+y+z|\}.
\end{align*}
Then 
\begin{align*}
|\delta_yp_3(t,x)|&\le C_1 e^{-C_2 a^2/t} t^{-(d+1)/2} |y|,\\
|\delta_y\delta_zp_3(t,x)|&\le C_1e^{-C_2b^2/t} t^{-(d+1)/2}\min\{|y|, |z|, t^{-1/2}|y||z|\}.
\end{align*}
\end{lmm}
\begin{proof}
Let $u(q) := e^{-|q|^2}$ for $q\in \rr^d$, so that 
\[
p_3(t,x) = (4\pi\beta t)^{-d/2} u((4\beta t)^{-1/2} x). 
\]
Let $v$ be any derivative of $u$, of any order. Clearly, $v(q)$ is some polynomial in $q$ times $u(q)$, and therefore $|v(q)|\le C_1 e^{-C_2|q|^2}$ for all $q\in \rr^d$, where $C_1$ and $C_2$ are some positive constants depending on $v$. In particular, the maximum value of $|v|$ on a line joining two points $q$ and $q'$ is bounded above by $C_1 e^{-C_2\min\{|q|^2,|q'|^2\}}$. This gives
\begin{align*}
|\delta_y p_3(t,x)| &=  (4\pi\beta t)^{-d/2} |u((4\beta t)^{-1/2} (x+y)) -  u((4\beta t)^{-1/2} x)|\notag \\
&\le C_1t^{-(d+1)/2} |y| e^{-C_2a^2/t}.
\end{align*}
This proves the first claim of the lemma. For the second claim, let us assume without loss of generality that $|y|\le |z|$. Note that by the above inequality,
\begin{align}
|\delta_y \delta_z p_3(t,x)| &\le |\delta_y p_3(t,x+z)| + |\delta_y p_3(t,x)|\notag \\
&\le C_1 e^{-C_2b^2/t}t^{-(d+1)/2} |y|.\label{p31}
\end{align}
On the other hand, by Taylor approximation and the above observation about the derivatives of $u$,
\begin{align*}
&|\delta_yp_3(t,x) - (4\pi \beta t)^{-d/2} (4\beta t)^{-1/2}y\cdot \nabla u((4\beta t)^{-1/2} x)|\\
 &\le C_1t^{-(d+2)/2}|y|^2 e^{-C_2\min\{|x+y|^2, |x|^2\}/t}, 
\end{align*}
and the same inequality holds with $x$ replaced by $x+z$ on both sides. This gives
\begin{align}
|\delta_y \delta_zp_3(t,x)| &= |\delta_yp_3(t,x+z)-\delta_yp_3(t,x)| \notag \\
&\le Ct^{-(d+1)/2} |y| |\nabla u((4\beta t)^{-1/2} (x+z)) - \nabla u((4\beta t)^{-1/2} x)| \notag \\
&\qquad + C_1 e^{-C_2b^2/t} t^{-(d+2)/2}|y|^2\notag \\
&\le C_1 e^{-C_2b^2/t}  t^{-(d+2)/2} (|y||z|+|y|^2)\notag \\
&\le C_1 e^{-C_2b^2/t}  t^{-(d+2)/2} |y||z|,\label{p32}
\end{align}
where we used $|y|\le |z|$ in the last step. Combining \eqref{p31} and \eqref{p32} completes the proof of the second claim.
\end{proof}
Finally, we need the following crude bound on $p(t,x)$ to deal with values of $x$ whose norms are `too large'.
\begin{lmm}\label{plmm}
For any $t\in \zz_{\ge 1}$ and $x\in \zz^d$, 
\begin{align*}
p(t,x) &\le C_1 e^{-C_2|x|^2/t}. 
\end{align*}
\end{lmm}
\begin{proof}
Take any $t\in \zz_{\ge 1}$ and $x = (x_1,\ldots,x_d)\in \zz^d$. Without loss of generality, suppose that $|x_1|\ge |x_i|$ for all $i$. Consider the evolution of the first coordinate of our random walk. At each time step, the first coordinate increases by one with probability $\beta$, decreases by one with probability $\beta$, and remains the same with probability $1-2\beta$. Thus, it evolves like a sum of i.i.d.~$\{-1,0,1\}$-valued random variables. It is now easy to see by Hoeffding's inequality for sums of i.i.d.~bounded random variables that at time $t$, the probability that the first coordinate equals $x_1$ is at most $C_1 e^{-C_2x_1^2/t}$. Since $|x_1|\ge |x_i|$ for all $i$, we have $x_1^2\ge |x|^2/d$. This completes the proof of the lemma. 
\end{proof} 

We are now ready to obtain the required estimates for the first- and second-order discrete derivatives of $p$. Proposition \ref{delprop} gives the bound on first-order derivatives, and following that, Proposition \ref{pprop} gives the bound on second order derivatives.
\begin{prop}\label{delprop}
Take any $t\in \zz_{\ge 1}$ and $x,y\in \zz^d$. Let $a := \min\{|x|, |x+y|\}$. Then we have
\begin{align*}
|\delta_yp(t,x)|&\le C_1 e^{-C_2 a^2/t} t^{-(d+1)/2} |y| + C_1t^{-(d+3)/2}(\log t)^{C_2}|y|. 
\end{align*}
Moreover, we also have $|\delta_y p(t,x)|\le C_1 e^{-C_2a^2/t}$. 
\end{prop}
\begin{proof}
The bounds hold trivially if $y=0$. So let us assume that $y\ne 0$. Again, the bounds hold trivially if $t=1$. So let us also assume that $t\ge 2$. By Lemma \ref{pp1lmm} and the triangle inequality,
\begin{align*}
&|\delta_yp(t,x)-\delta_yp_1(t,x)| \\
&\le |p(t,x+y)-p_1(t,x+y)| + |p(t,x)-p_1(t,x)|\\
&\le C_1 e^{-C_2(\log t)^2}.
\end{align*}
Similarly, by Lemma \ref{p2p3lmm},
\[
|\delta_y p_2(t,x)-\delta_y p_3(t,x)| \le C_1 e^{-C_2(\log t)^2}.
\]
Next, by Lemma \ref{p1p2lmm},
\begin{align*}
|\delta_y p_1(t,x)-\delta_y p_2(t,x)|&\le C_1|y|t^{-(d+3)/2}(\log t)^{C_2}.
\end{align*}
Combining these with the bound on $|\delta_y p_3(t,x)|$ from Lemma \ref{p3lmm}, and observing that $ C_1|y|t^{-(d+3)/2}(\log t)^{C_2}$ dominates $C_1 e^{-C_2(\log t)^2}$ since $y\ne 0$ and $t\ge 2$,  proves the first claim of the proposition. The second claim follows  by Lemma~\ref{plmm} and the triangle inequality.
\end{proof}

\begin{prop}\label{pprop}
Take any $t\in \zz_{\ge 1}$ and $x,y,z\in \zz^d$. Let 
\[
b := \min\{|x|, |x+y|, |x+z|, |x+y+z|\}.
\]
Then we have
\begin{align*}
|\delta_y\delta_zp(t,x)|&\le C_1e^{-C_2b^2/t}t^{-(d+1)/2}\min\{|y|, |z|, t^{-1/2}|y||z|\}\\
&\qquad + C_1|y||z| t^{-(d+4)/2}(\log t)^{C_2}. 
\end{align*}
Moreover, we also have $|\delta_y \delta_zp(t,x)|\le C_1 e^{-C_2b^2/t}$. 
\end{prop}
\begin{proof}
The bound holds trivially if either $y=0$ or $z=0$, because in that case, $\delta_y \delta_zp(t,x)=0$. So let us assume that $y$ and $z$ are both nonzero. Then, the bound holds trivially if $t=1$. So let us further assume that $t\ge 2$. By Lemma \ref{pp1lmm} and the triangle inequality,
\begin{align*}
&|\delta_y\delta_z p(t,x)-\delta_y\delta_zp_1(t,x)| \\
&\le |p(t,x+y+z)-p_1(t,x+y+z)| + |p(t,x+y)-p_1(t,x+y)| \\
&\qquad + |p(t,x+z)-p_1(t,x+z)| + |p(t,x)-p_1(t,x)|\\
&\le C_1 e^{-C_2(\log t)^2}.
\end{align*}
Similarly, by Lemma \ref{p2p3lmm},
\[
|\delta_y\delta_z p_2(t,x)-\delta_y\delta_zp_3(t,x)| \le C_1 e^{-C_2(\log t)^2}.
\]
Next, by Lemma \ref{p1p2lmm},
\begin{align*}
|\delta_y\delta_zp_1(t,x)-\delta_y\delta_zp_2(t,x)|&\le C_1|y||z| t^{-(d+4)/2}(\log t)^{C_2}.
\end{align*}
Combining these with the bound on $|\delta_y\delta_z p_3(t,x)|$ from Lemma \ref{p3lmm}, and the observation that $C_1|y||z| t^{-(d+4)/2}(\log t)^{C_2}$ dominates $C_1 e^{-C_2(\log t)^2}$ since $y$ and $z$ are nonzero and $t\ge 2$, proves the first claim of the proposition. The second one is a consequence of Lemma \ref{plmm} and the triangle inequality.
\end{proof}

We now have the necessary bounds for spatial derivatives of $p$. We will also need bounds for first-order derivatives with respect to $t$, and space-times mixed derivatives. For $r\in \zz_{\ge 0}$ and a function $w:\zz_{\ge 0} \times \zz^d\to \rr$, define the temporal derivative
\[
\delta'_rw(t,x) := w(t+r,x)-w(t,x).
\]
Our goal is to get bounds on quantities like $\delta_r'p(t,x)$ and $\delta_r'\delta_y p(t,x)$. The first step, as before, is to get the corresponding bounds for $p_3$ instead of $p$. 
\begin{lmm}\label{p3lmm2}
For any two integers $t\ge r\ge1$, and any $x\in \zz^d$, 
\begin{align*}
|\delta_r'p_3(t,x)| &\le C_1 e^{-C_2|x|^2/t} rt^{-(d+2)/2}.
\end{align*}
\end{lmm}
\begin{proof}
Treating $t$ as a continuous variable in the formula \eqref{p3form} for $p_3$, let $p_3'$ denote its derivative with respect to $t$. Then note that 
\begin{align*}
|p_3'(t,x)| &\le C_1 t^{-(d+2)/2} (1+t^{-1}|x|^2) e^{-C_2|x|^2/t}\\
&\le C_1 t^{-(d+2)/2} e^{-C_2|x|^2/t}. 
\end{align*}
Since $r\le t$, the maximum value of this derivative in the time interval $[t,t+r]$ is bounded by $C_1 t^{-(d+2)/2} e^{-C_2|x|^2/t}$. By the mean value theorem, this completes the proof of the lemma. 
\end{proof}
\begin{lmm}\label{p3lmm3}
Take any two integers $t\ge r\ge 1$, and any $x,y\in \zz^d$. Let $a := \min\{|x|, |x+y|\}$. Then 
\begin{align*}
|\delta_r'\delta_yp_3(t,x)|&\le C_1 e^{-C_2 a^2/t} |y| r t^{-(d+3)/2}. 
\end{align*}
\end{lmm}
\begin{proof}
In this proof we will use $K$ as shorthand for $C_1 e^{-C_2a^2/t}$. Fixing $x$ and $y$, let 
\[
w(t):= \delta_yp_3(t,x)
\]
and let $w'$ denote its derivative with respect to $t$, considering $t$ as a continuous parameter. Let $u(q) := e^{-|q|^2}$ and $v(q) := q\cdot \nabla u(q)$ for $q\in \rr^d$. Then
\begin{align*}
w'(t) &= -\frac{d}{2}(4\pi \beta)^{-d/2} t^{-(d+2)/2} (u((4\beta t)^{-1/2} (x+y)) - u((4\beta t)^{-1/2} x))\\
&\qquad -\frac{1}{2} (4\pi \beta t)^{-d/2} t^{-1} (v((4\beta t)^{-1/2} (x+y)) - v((4\beta t)^{-1/2} x)).
\end{align*}
It is not hard to see that $|\nabla v(q)|\le C_1 e^{-C_2 |q|^2}$ for all $q$. Thus,
\begin{align*}
|v((4\beta t)^{-1/2} (x+y)) - v((4\beta t)^{-1/2} x)| &\le Kt^{-1/2}|y|.
\end{align*}
Similarly, 
\begin{align*}
|u((4\beta t)^{-1/2} (x+y)) - u((4\beta t)^{-1/2} x)|&\le Kt^{-1/2}|y|.
\end{align*}
Combining, we get
\begin{align*}
|w'(t)| &\le Kt^{-(d+3)/2}|y|.
\end{align*}
Since $|w(t+r)-w(t)|\le r \max_{t\le s\le t+r} |w'(s)|$ and $r\le t$, this shows that
\begin{align*}
|\delta_r' \delta_yp_3(t,x)| &\le Kt^{-(d+3)/2}|y|r.
\end{align*}
This completes the proof of the lemma.
\end{proof}
Now that we have the required bounds for the derivatives of $p_3$, we can use them to deduce the corresponding bounds for the derivatives of $p$. 
\begin{prop}\label{plmm2}
For any two integers $t\ge r\ge1$, and any $x\in \zz^d$, 
\begin{align*}
|\delta_r'p(t,x)| &\le C_1e^{-C_2|x|^2/t} rt^{-(d+2)/2} + C_1 t^{-(d+2)/2}(\log t)^{C_2}.
\end{align*}
Moreover, we also have $|\delta_r'p(t,x)|\le C_1 e^{-C_2|x|^2/t}$. 
\end{prop}
\begin{proof}
Since $1\le r\le t$, it is easy to see that the case $t=1$ is trivial. So let us assume that $t\ge 2$. First, note that by Lemma \ref{pp1lmm}, 
\begin{align*}
|\delta_r'p(t,x)-\delta_r'p_1(t,x)| &\le |p(t+r,x)-p_1(t+r,x)| + |p(t,x)-p_1(t,x)|\\
&\le C_1 e^{-C_2(\log t)^2}.
\end{align*}
Similarly, by Lemma \ref{p2p3lmm}, 
\[
|\delta_r'p_2(t,x)-\delta_r'p_3(t,x)| \le C_1 e^{-C_2(\log t)^2}.
\]
Next, by Lemma \ref{p1p2lmm},
\begin{align*}
&|\delta_r'p_1(t,x)-\delta_r'p_2(t,x)| \\
&\le |p_1(t+r,x)-p_2(t+r,x)| + |p_1(t,x)-p_2(t,x)|\\
&\le C_1 t^{-(d+2)/2}(\log t)^{C_2}. 
\end{align*}
Combining the above inequalities with the bound on $|\delta_r'p_3(t,x)|$ from Lemma \ref{p3lmm2}, and the observation that $C_1 t^{-(d+2)/2}(\log t)^{C_2}$ dominates $C_1 e^{-C_2(\log t)^2}$ since $t\ge 2$, we  get the first inequality in the statement of the lemma. The second one follows from Lemma \ref{plmm} and the triangle inequality.
\end{proof}
\begin{prop}\label{plmm3}
Take any two integers $t\ge r\ge 1$, and any $x, y\in \zz^d$. Let $a:=\min\{|x|, |x+y|\}$. Then  
\begin{align*}
|\delta_r'\delta_yp(t,x)| &\le C_1e^{-C_2a^2/t} t^{-(d+3)/2}r|y| +  C_1 |y| t^{-(d+3)/2}(\log t)^{C_2}.
\end{align*}
Moreover, we also have $|\delta_r'\delta_yp(t,x)|\le C_1 e^{-C_2a^2/t}$. 
\end{prop}
\begin{proof}
The bounds hold trivially if $y=0$. So let us assume that $y\ne 0$. Then again, the bounds are trivial if $t=1$. So let us also assume that $t\ge 2$. First, note that by Lemma \ref{pp1lmm}, 
\begin{align*}
&|\delta_r'\delta_yp(t,x)-\delta_r'\delta_yp_1(t,x)| \\
&\le |p(t+r,x+y)-p_1(t+r,x+y)| + |p(t,x+y)-p_1(t,x+y)|\\
&\qquad + |p(t+r, x)-p_1(t+r,x)| + |p(t,x)-p_1(t,x)|\\
&\le C_1 e^{-C_2(\log t)^2}.
\end{align*}
Similarly, by Lemma \ref{p2p3lmm}, 
\[
|\delta_r'\delta_yp_2(t,x)-\delta_r'\delta_yp_3(t,x)| \le C_1 e^{-C_2(\log t)^2}.
\]
Next, by Lemma \ref{p1p2lmm},
\begin{align*}
&|\delta_r'\delta_yp_1(t,x)-\delta_r'\delta_yp_2(t,x)| \\
&\le |\delta_yp_1(t+r,x)-\delta_yp_2(t+r,x)| + |\delta_yp_1(t,x)-\delta_yp_2(t,x)|\\
&\le C_1|y|t^{-(d+3)/2}(\log t)^{C_2}. 
\end{align*}
Combining the above inequalities with the bound on $|\delta_r'\delta_yp_3(t,x)|$ from Lemma \ref{p3lmm3}, and the observation that $C_1|y|t^{-(d+3)/2}(\log t)^{C_2}$ dominates $C_1 e^{-C_2(\log t)^2}$ since $y\ne 0$ and $t\ge 2$, we get the first inequality in the statement of the lemma. The second one follows from Lemma~\ref{plmm} and the triangle inequality.
\end{proof}

\subsection{Smoothness of the discrete surface}\label{surfacederivsec}
In this subsection we will derive bounds on the first- and second-order discrete derivatives of $f_\ve$ using the bounds on the derivatives of $p$ obtained in the previous subsection, the recursive formula from Proposition~\ref{fformlmm}, and the uniform bound on $h_\ve$ from Lemma \ref{lip2lmm}. Our first result gives a bound on second-order spatial derivatives of $f_\ve$. 
\begin{prop}\label{secondderiv}
Take any $t\in \zz_{\ge1}$ and any $x,y,z\in \zz^d$ with $|y|\le |z|$. Then 
\begin{align*}
&|\delta_y \delta_z f_\ve(t,x)| \\
&\le C_1 \ve^2 |y| |z|\log (2+t/|z|^2) \\
&\qquad + C_1(1+\ve|x| +\ve|y|+\ve|z|+\ve t^{1/2})t^{-1/2}|y|\min\{1, t^{-1/2}|z|\} \notag\\
&\qquad + C_1(1+\ve |x| +\ve|y|+\ve|z|+ \ve t^{1/2}\log t) |y||z| t^{-2}(\log t)^{C_2}.
\end{align*}
\end{prop}
\begin{proof}
If $y=0$ or $z=0$, the left side is zero and hence the bound holds trivially. So let us assume that $y$ and $z$ are nonzero. 
By Proposition \ref{fformlmm}, we get
\begin{align}
\delta_y\delta_zf_\ve (t,x) &= \sum_{w\in \zz^d} \delta_y\delta_zp(t,x-w) g_\ve(w) \notag\\
&\qquad + \sum_{0\le s\le t-1} \sum_{w\in \zz^d} \delta_y\delta_z p(s, x-w) h_\ve(t-s,w).\label{propeq1}
\end{align}
Take any $2\le s\le t-1$ (which exists only if $t\ge 3$; otherwise, we do not have to worry about such $s$). By Lemma \ref{lip2lmm}, $|h_\ve(t-s,w)|\le C\ve^2$ for all $w$. Thus,
\begin{align}
\biggl| \sum_{w\in \zz^d} \delta_y\delta_z p(s, x-w) h_\ve(t-s,w)\biggr| &\le C\ve^2 \sum_{w\in \zz^d} |\delta_y\delta_z p(s, x-w)|.\label{wineq0}
\end{align}
Let $D$ be the set of all $w$ such that at least one of the quantities $|x-w|$, $|x-w+y|$, $|x-w+z|$ and $|x-w+y+z|$ is less than $\sqrt{s} \log s$.  Then by the second bound from Proposition \ref{pprop}, we get
\begin{align}
\sum_{w\notin D} |\delta_y\delta_z p(s, x-w)|&\le C_1\sum_{|v|\ge \sqrt{s}\log s} e^{-C_2|v|^2/s}\notag \\
&\le C_1 e^{-C_2(\log s)^2}. \label{wineq1}
\end{align}
Fixing $x$, $y$ and $z$, define
\[
q(w) := \min\{|x-w|, |x-w+y|, |x-w+z|, |x-w+y+z|\}
\]
for $w\in \zz^d$. Now, it is not hard to see that for any $C$, 
\begin{align}\label{ecc}
\sum_{v\in \zz^d} e^{-C|v|^2/s} \le C' s^{d/2}
\end{align}
where $C'$ depends on $C$ and $d$. This implies that
\begin{align*}
&\sum_{w\in \zz^d} e^{-Cq(w)^2/s} \\
&\le \sum_{w\in \zz^d} (e^{-C|x-w|^2/s} + e^{-C|x-w+y|^2/s} + e^{-C|x-w+z|^2/s} + e^{-C|x-w+y+z|^2/s})\\
&= 4\sum_{w\in \zz^d} e^{-C|w|^2/s} \le C' s^{d/2},
\end{align*}
where $C'$ depends on $C$ and $d$. Also, note that $|D|\le Cs^{d/2}(\log s)^d$. Therefore, the first bound from Proposition \ref{pprop} gives 
\begin{align}
\sum_{w\in D} |\delta_y\delta_z p(s, x-w)| &\le C_1 s^{-(d+1)/2}|y|\min\{1, s^{-1/2}|z|\}\sum_{w\in D} e^{-C_2q(w)^2/s}\notag \\
&\qquad + C_1 |y||z|s^{-(d+4)/2} (\log s)^{C_2} |D| \notag \\
&\le C_1 s^{-1/2}|y|\min\{1, s^{-1/2}|z|\}\notag \\
&\qquad + C_1 |y||z|s^{-2} (\log s)^{C_2}. \label{wineq2}
\end{align}
Combining \eqref{wineq0}, \eqref{wineq1} and \eqref{wineq2}, and the assumption that $y$ and $z$ are nonzero (which allows the bound from \eqref{wineq2} to dominate the bound from \eqref{wineq1}), we have
\begin{align*}
&\biggl| \sum_{w\in \zz^d} \delta_y\delta_z p(s, x-w) h_\ve(t-s,w)\biggr| \\
&\le  C_1\ve^2 |y| s^{-1/2}\min\{1, s^{-1/2}|z|\}  + C_1\ve^2 |y||z|s^{-2} (\log s)^{C_2}.
\end{align*}
Consequently,
\begin{align*}
&\biggl|\sum_{2\le s\le t-1} \sum_{w\in \zz^d} \delta_y\delta_z p(s, x-w) h_\ve (t-s,w)\biggr|\\
&\le  C_1\ve^2|y|\sum_{2\le s\le t-1} s^{-1/2}\min\{1, s^{-1/2}|z|\} + C_1 \ve^2 |y||z|. 
\end{align*}
Now note that
\begin{align*}
\sum_{2\le s\le t-1} s^{-1/2}\min\{1, s^{-1/2}|z|\}  &\le \sum_{2\le s\le |z|^2} s^{-1/2} + \sum_{|z|^2< s\le t-1} s^{-1}|z|\\
&\le   
\begin{cases}
C_1 |z|\log(t/|z|^2) &\text{ if } |z|< \sqrt{t},\\
C_1 |z| &\text{ if } |z|\ge \sqrt{t}. 
\end{cases}
\end{align*}
Lastly, note that by the uniform bound on $h_\ve$ from Lemma \ref{lip2lmm}, the contributions from $s=0$ and $s=1$ in \eqref{propeq1} are bounded by $C\ve^2$. Combining these observations, we get
\begin{align}
&\biggl|\sum_{0\le s\le t-1} \sum_{w\in \zz^d} \delta_y\delta_z p(s, x-w) h_\ve(t-s,w)\biggr| \notag\\
&\le C_1 \ve^2 |y| |z|\log (2+t/|z|^2). \label{propeq2}
\end{align}
This takes care of the second term on the right side of \eqref{propeq1}. Let us now turn our attention to the first term. By the Lipschitz property of $g$, we have that 
\begin{align}\label{gsize}
|g_\ve(w)|\le C(1+\ve |w|)
\end{align}
for all $w\in \zz^d$. Let $D$ be as before, with $s$ replaced  by $t$. Then by the second bound from Proposition \ref{pprop}, the bound \eqref{gsize}, and the assumption that $\ve <1$, we have
\begin{align}
\biggl|\sum_{w\notin D} \delta_y\delta_z p(t, x-w)g_\ve(w) \biggr|&\le C_1 \sum_{|v|\ge \sqrt{t}\log t} (1+\ve |x|+\ve|v|)e^{-C_2|v|^2/t}\notag \\
&\le C_1(1+\ve |x|) e^{-C_2(\log t)^2}. \label{dineq1}
\end{align}
Now recall the inequality \eqref{ecc}. Analogous to that, it is also easy to show that 
\begin{align}\label{ecc2}
\sum_{v\in \zz^d} |v|e^{-C|v|^2/t} &\le C' t^{(d+1)/2}.
\end{align}
Let the function $q$ be defined as before. Using  the first bound from Proposition~\ref{pprop}, the inequalities \eqref{gsize}, \eqref{ecc}, and \eqref{ecc2}, and the facts that $|y|\le |z|$ and $|D|\le Ct^{d/2}(\log t)^d$, we get
\begin{align}
&\biggl|\sum_{w\in D} \delta_y\delta_z p(t, x-w)g_\ve(w) \biggr|\notag\\
&\le C_1t^{-(d+1)/2}\min\{|y|, |z|, t^{-1/2}|y||z|\}\sum_{w\in D} (1+\ve|w|)e^{-C_2q(w)^2/t}\notag \\
&\qquad + C_1|y||z| t^{-(d+4)/2}(\log t)^{C_2} \sum_{w\in D} (1+\ve|w|)\notag \\
&\le C_1(1+\ve|x| +\ve|y|+\ve|z|+ \ve t^{1/2})t^{-1/2}|y|\min\{1, t^{-1/2}|z|\} \notag\\
&\qquad + C_1 (1+\ve |x|+\ve|y|+\ve|z| + \ve t^{1/2}\log t)|y||z| t^{-2}(\log t)^{C_2}. \label{dineq2}
\end{align}
Since $y$ and $z$ are nonzero, $t^{-1/2}|y|\min\{1, t^{-1/2}|z|\}$ dominates $e^{-C_2(\log t)^2}$. Thus, combining \eqref{dineq1} and \eqref{dineq2}, we get
\begin{align}
&\biggl|\sum_{w\in \zz^d} \delta_y\delta_z p(t, x-w)g_\ve(w) \biggr|\notag \\
&\le C_1(1+\ve|x| +\ve|y|+\ve|z|+ \ve t^{1/2})t^{-1/2}|y|\min\{1, t^{-1/2}|z|\} \notag\\
&\qquad + C_1(1+\ve |x| +\ve|y|+\ve|z|+ \ve t^{1/2}\log t) |y||z| t^{-2}(\log t)^{C_2}. \label{propeq3}
\end{align}
Combining \eqref{propeq1}, \eqref{propeq2} and \eqref{propeq3} completes the proof.
\end{proof}
The next result gives a bound for the first-order temporal derivatives of $f_\ve$.  The scheme of the proof is just as before, using the recursive formula for $f_\ve$ from Proposition \ref{fformlmm} and the bounds on the temporal derivatives of $p$ from the previous subsection.
\begin{prop}\label{timeprop1}
For any two integers $t\ge r\ge 1$ and any $x\in \zz^d$,
\begin{align*}
|\delta_r'f_\ve(t,x)| &\le  C_1\ve^2 r(1+\log(t/r)) + C_1\ve^2(\log t)^{C_2} \\
&\qquad + C_1 rt^{-1} (1+\ve |x|+\ve t^{1/2})  \\
&\qquad + C_1 (1+\ve |x| + \ve t^{1/2}\log t) t^{-1}(\log t)^{C_2}.
\end{align*}
\end{prop}
\begin{proof}
By Proposition \ref{fformlmm}, 
\begin{align}
\delta_r'f_\ve(t,x) &= \sum_{w\in \zz^d} \delta_r'p(t,x-w) g_\ve(w) \notag \\
&\qquad + \sum_{0\le s\le t-1} \sum_{w\in \zz^d} \delta_r'p(s, x-w) h_\ve(t-s,w)\notag \\
&\qquad + \sum_{0\le s\le r-1} \sum_{w\in \zz^d}p(s, x-w) h_\ve(t+r-s, w). \label{prop2eq0}
\end{align}
Recall that by Lemma \ref{lip2lmm}, $|h_\ve(s,w)|\le C\ve^2$ for all $s$ and $w$. Therefore, 
\begin{align}
&\biggl|\sum_{0\le s\le r-1} \sum_{w\in \zz^d}p(s, x-w) h_\ve (t+r-s, w)\biggr| \notag\\
&\le C\ve^2 \sum_{0\le s\le r-1}   \sum_{w\in \zz^d}p(s, x-w)= C\ve^2 r. \label{prop2eq1}
\end{align}
Similarly, we have 
\begin{align}
&\biggl|\sum_{0\le s\le r} \sum_{w\in \zz^d} \delta_r'p(s, x-w) h_\ve(t-s,w)\biggr|\notag \\
&\le  C\ve^2 \sum_{0\le s\le r} \sum_{w\in \zz^d} (p(s+r, x-w)  + p(s,x-w)) \notag\\
&\le C\ve^2 r.\label{cve1}
\end{align}
Finally, take any $r+1\le s\le t-1$. (If $t\le r+1$, then no such $s$ exists and we do not have to worry about this case.) Let $D := \{w:|x-w|\le  \sqrt{s}\log s\}$. Then by the second bound from Lemma \ref{plmm2},
\begin{align*}
\biggl| \sum_{w\notin D} \delta_r'p(s, x-w) h_\ve(t-s,w)\biggr| &\le C_1\ve^2 \sum_{w\notin D} e^{-C_2|x-w|^2/s}\\
&\le C_1 \ve^2 e^{-C_2(\log s)^2}. 
\end{align*}
But, by the first bound from Proposition \ref{plmm2} and the inequality \eqref{ecc},
\begin{align*}
&\biggl| \sum_{w\in D} \delta_r'p(s, x-w) h_\ve(t-s,w)\biggr|\\
&\le C_1 \ve^2\sum_{w\in D}  r s^{-(d+2)/2}e^{-C_2|x-w|^2/s}  +  C_1\ve^2 |D| s^{-(d+2)/2}(\log s)^{C_3}\\
&\le C_1 \ve^2 r s^{-1} + C_1 \ve^2 s^{-1}(\log  s)^{C_2}. 
\end{align*}
Adding the last two displays, we get that for any $r+1\le s\le t-1$, 
\begin{align}
\biggl| \sum_{w\in\zz^d} \delta_r'p(s, x-w) h_\ve(t-s,w)\biggr| \le  C_1 \ve^2 r s^{-1} + C_1 \ve^2 s^{-1}(\log  s)^{C_2}. \label{cve2}
\end{align}
Combining \eqref{cve1} and \eqref{cve2}, we get
\begin{align}
&\biggl|\sum_{0\le s\le t-1} \sum_{w\in \zz^d} \delta_r'p(s, x-w) h_\ve(t-s,w)\biggr| \notag \\
&\le C_1 \ve^2 r(1+ \log (t/r)) + C_1 \ve^2 (\log t)^{C_2}.\label{prop2eq2}
\end{align}
Next, let $D:=\{w: |x-w|\le  \sqrt{t}\log t\}$. Then  by the second bound from Proposition~\ref{plmm2} and the facts that $r\le t$ and $|g_\ve(w)|\le C(1+\ve|w|)$,  we get
\begin{align*}
\biggl| \sum_{w\notin D} \delta_r'p(t,x-w) g_\ve(w)\biggr| &\le C_1 \sum_{w\notin D} (1+\ve|w|) e^{-C_2|x-w|^2/t} \\
&\le C_1(1+\ve|x|) e^{-C_2(\log t)^2}.
\end{align*}
Similarly, by the first bound from Proposition \ref{plmm2}, 
\begin{align*}
&\biggl| \sum_{w\in D} \delta_r'p(t,x-w) g_\ve(w)\biggr| \\
&\le C_1\sum_{w\in D} (1+\ve|w|) r t^{-(d+2)/2}e^{-C_2|x-w|^2/t}\\
&\qquad +  C_1 t^{-(d+2)/2}(\log t)^{C_2}\sum_{w\in D} (1+\ve|w|)\\
&\le C_1 rt^{-1} (1+\ve |x|+\ve t^{1/2}) \\
&\qquad + C_1 (1+\ve |x| + \ve t^{1/2}\log t) t^{-1}(\log t)^{C_2}. 
\end{align*}
Since $t^{-1}(\log t)^{C_1}$ dominates $e^{-C_2(\log t)^2}$, combining the last two displays gives
\begin{align}
&\biggl| \sum_{w\in \zz^d} \delta_r'p(t,x-w) g_\ve(w)\biggr|\notag \\
 &\le C_1 rt^{-1} (1+\ve |x|+\ve t^{1/2}) \notag \\
&\qquad + C_1 (1+\ve |x| + \ve t^{1/2}\log t) t^{-1}(\log t)^{C_2}. \label{prop2eq3}
\end{align}
Combining \eqref{prop2eq0}, \eqref{prop2eq1}, \eqref{prop2eq2} and \eqref{prop2eq3} completes the proof.
\end{proof}
The next result gives bounds for mixed space-time derivatives of $f_\ve$. Again, the proof uses Proposition \ref{fformlmm} and the bounds on mixed derivatives of $p$ from the previous subsection.
\begin{prop}\label{timeprop2}
For any two integers $t\ge r\ge 1$ and  any $x,y\in \zz^d$,
\begin{align*}
|\delta_r'\delta_yf_\ve(t,x) | &\le  C_1 \ve^2 \sqrt{r}|y| + C_1 \ve^2 |y| r^{-1/2}(\log t)^{C_2} \\
&\qquad + C_1 rt^{-3/2}|y| (1+\ve |x|+\ve|y| + \ve t^{1/2}) \notag \\
&\qquad + C_1|y| (1+\ve |x| + \ve|y| + \ve t^{1/2}\log t) t^{-3/2}(\log t)^{C_2}.
\end{align*}
\end{prop}
\begin{proof}
The bound holds trivially if $y=0$. So let us assume that $y\ne 0$. By Proposition \ref{fformlmm}, 
\begin{align}
\delta_r'\delta_yf_\ve(t,x) &= \sum_{w\in \zz^d} \delta_r'\delta_yp(t,x-w) g_\ve(w) \notag \\
&\qquad + \sum_{0\le s\le t-1} \sum_{w\in \zz^d} \delta_r'\delta_yp(s, x-w) h_\ve(t-s,w)\notag \\
&\qquad + \sum_{0\le s\le r-1} \sum_{w\in \zz^d}\delta_y p(s, x-w) h_\ve(t+r-s, w). \label{prop3eq0}
\end{align}
By Lemma \ref{lip2lmm}, 
\begin{align*}
&\biggl|\sum_{0\le s\le r-1} \sum_{w\in \zz^d}\delta_y p(s, x-w) h_\ve(t+r-s, w)\biggr| \\
&\le C\ve^2 \sum_{0\le s\le r-1}   \sum_{w\in \zz^d}|\delta_y p(s, x-w)|. 
\end{align*}
Take any $2\le s\le r-1$. Let $D$ be the set of all $w$ such the minimum of $|x-w|$ and $|x+y-w|$ is less than $\sqrt{s}\log s$. Then by the second bound from Proposition~\ref{delprop}, we get
\begin{align*}
\sum_{w\notin D}|\delta_y p(s, x-w)| &\le C_1 e^{-C_2(\log s)^2}. 
\end{align*}
On the other hand, by the first bound from Proposition \ref{delprop} and the inequality \eqref{ecc}, we have
\begin{align*}
\sum_{w\in D}|\delta_y p(s, x-w)| &\le  C_1  s^{-(d+1)/2} |y| \sum_{w\in D} (e^{-C_2 |x-w|^2/s}+ e^{-C_2 |x+y-w|^2/s}) \\
&\qquad + C_1|y|s^{-(d+3)/2}(\log s)^{C_3}|D|\\
&\le C_1 s^{-1/2}|y| + C_1|y|s^{-3/2}(\log s)^{C_2}. 
\end{align*}
Combining the last three displays, we get
\begin{align*}
\biggl|\sum_{2\le s\le r-1} \sum_{w\in \zz^d}\delta_y p(s, x-w) h_\ve(t+r-s, w)\biggr| &\le C \ve^2 \sqrt{r} |y|. 
\end{align*}
Since the contributions from the $s=0$ and $s=1$ terms amount to at most $C\ve^2$, we have
\begin{align}
\biggl|\sum_{0\le s\le r-1} \sum_{w\in \zz^d}\delta_y p(s, x-w) h_\ve(t+r-s, w)\biggr| &\le C \ve^2 \sqrt{r} |y|. \label{prop3eq1}
\end{align}
Similarly,
\begin{align}
&\biggl|\sum_{0\le s\le r} \sum_{w\in \zz^d} \delta_r'\delta_yp(s, x-w) h_\ve(t-s,w)\biggr|\notag\\
&\le C\ve^2\sum_{0\le s\le r} \sum_{w\in \zz^d} (|\delta_yp(s+r, x-w)| + |\delta_y p(s,x-w)|)\notag\\
&\le C\ve^2 \sqrt{r} |y|.\label{pcve1}
\end{align}
Finally, take any $r+1\le s\le t-1$. As before, let $D$ be the set of all $w$ such the minimum of $|x-w|$ and $|x+y-w|$ is less than $\sqrt{s}\log s$. Then by Lemma \ref{lip2lmm} and the second bound from Proposition \ref{plmm3},
\begin{align*}
\biggl| \sum_{w\notin D} \delta_r'\delta_yp(s, x-w) h_\ve(t-s,w)\biggr| &\le C_1 \ve^2 e^{-C_2(\log s)^2}. 
\end{align*}
On the other hand, by the first bound from Proposition \ref{plmm3} and the inequality \eqref{ecc},
\begin{align*}
&\biggl| \sum_{w\in D} \delta_r'\delta_yp(s, x-w) h_\ve (t-s,w)\biggr|\\
&\le C_1 \ve^2|y|\sum_{w\in D}  r s^{-(d+3)/2}(e^{-C_2|x-w|^2/s}+e^{-C_2|x+y-w|^2/s}) \\
&\qquad +  C_1\ve^2|y| |D| s^{-(d+3)/2}(\log s)^{C_3}  \\
&\le C_1 \ve^2 r|y| s^{-3/2} + C_1 \ve^2 |y| s^{-3/2}(\log  s)^{C_2}. 
\end{align*}
Adding the last two displays and using the assumption that $y\ne 0$, we get that for any $r+1\le s\le t-1$, 
\begin{align}
&\biggl| \sum_{w\in\zz^d} \delta_r'\delta_yp(s, x-w) h_\ve (t-s,w)\biggr| \notag \\
&\le  C_1 \ve^2 r |y|s^{-3/2} + C_1 \ve^2 |y|s^{-3/2}(\log  s)^{C_2}. \label{pcve2}
\end{align}
Combining \eqref{pcve1} and \eqref{pcve2}, we get
\begin{align}
&\biggl|\sum_{0\le s\le t-1} \sum_{w\in \zz^d} \delta_r'\delta_yp(s, x-w) h_\ve(t-s,w)\biggr| \notag \\
&\le C_1 \ve^2 \sqrt{r}|y| + C_1 \ve^2 |y| r^{-1/2}(\log t)^{C_2}.\label{prop3eq2}
\end{align}
Next, let $D$ be the set of all $w$ such the minimum of $|x-w|$ and $|x+y-w|$ is less than $\sqrt{t}\log t$. Then  by the second bound from Proposition~\ref{plmm3} and the facts that $r\le t$ and $|g_\ve(w)|\le C(1+\ve|w|)$ for all $w$,  we get
\begin{align*}
&\biggl| \sum_{w\notin D} \delta_r'\delta_yp(t,x-w) g_\ve(w)\biggr| \\
&\le C_1 \sum_{w\notin D} (1+\ve|w|) ( e^{-C_2|x-w|^2/t}+ e^{-C_2|x+y-w|^2/t}) \\
&\le C_1(1+\ve|x|+\ve |y|) e^{-C_2(\log t)^2}.
\end{align*}
Similarly, by the first bound from Proposition \ref{plmm3}, 
\begin{align*}
&\biggl| \sum_{w\in D} \delta_r'\delta_yp(t,x-w) g_\ve(w)\biggr| \\
&\le C_1|y|\sum_{w\in D} (1+\ve|w|) r t^{-(d+3)/2}( e^{-C_2|x-w|^2/t}+ e^{-C_2|x+y-w|^2/t})\\
&\qquad +  C_1 |y|t^{-(d+3)/2}(\log t)^{C_3}\sum_{w\in D} (1+\ve|w|)\\
&\le C_1 rt^{-3/2}|y| (1+\ve |x|+\ve|y| + \ve t^{1/2}) \\
&\qquad + C_1|y| (1+\ve |x| + \ve|y| + \ve t^{1/2}\log t) t^{-3/2}(\log t)^{C_2}. 
\end{align*}
Combining the last two displays, and observing that $|y| t^{-3/2}$  dominates $e^{-C(\log t)^2}$, we get
\begin{align}
&\biggl| \sum_{w\in \zz^d} \delta_r'\delta_yp(t,x-w) g_\ve(w)\biggr| \notag\\
&\le C_1 rt^{-3/2}|y| (1+\ve |x|+\ve|y| + \ve t^{1/2}) \notag \\
&\qquad + C_1|y| (1+\ve |x| + \ve|y| + \ve t^{1/2}\log t) t^{-3/2}(\log t)^{C_2}. \label{prop3eq3}
\end{align}
Combining \eqref{prop3eq0}, \eqref{prop3eq1}, \eqref{prop3eq2} and \eqref{prop3eq3}  completes the proof.
\end{proof}
We need one more estimate, to deal with situations where $t$ is small. In this case, $f_\ve(t,x)$ is expected to be close to $g_\ve(x)$. The following result makes this precise. 
\begin{prop}\label{bdrycont}
For any $t\in \zz_{\ge 1}$ and any $x\in \zz^d$,
\[
|f_\ve(t,x)-g_\ve(x)|\le C\ve \sqrt{t} + C\ve^2 t.
\]
\end{prop}
\begin{proof}
By Proposition \ref{fformlmm},
\begin{align*}
|f_\ve(t,x)-g_\ve(x)| &\le \sum_{y\in \zz^d} p(t,x-y) |g_\ve(y) -g_\ve(x)| \\
&\qquad + \sum_{0\le s\le t-1} \sum_{y\in \zz^d} p(s, x-y) |h_\ve(t-s,y)|.
\end{align*}
By the Lipschitz property of $g$ and simple properties of random walks,
\begin{align*}
\sum_{y\in \zz^d} p(t,x-y) |g_\ve(y) -g_\ve(x)|  &\le C\ve \sum_{y\in \zz^d} p(t,x-y) |x-y| \\
&\le C\ve \sqrt{t}. 
\end{align*}
On the other hand, by Lemma \ref{lip2lmm},
\begin{align*}
\sum_{0\le s\le t-1} \sum_{y\in \zz^d} p(s, x-y) |h_\ve(t-s,y)| &\le C\ve^2 \sum_{0\le s\le t-1} \sum_{y\in \zz^d} p(s, x-y) \\
&= C\ve^2 t. 
\end{align*}
Combining the last two displays completes the proof.
\end{proof}

\subsection{Existence of subsequential limits}\label{subseqsec}
We will now use the estimates from the previous subsection to show that the rescaled and renormalized function $f^{(\ve)}$ converges pointwise to a limit function along a subsequence. In fact, any subsequence will have a further subsequence that converges. Moreover, we will show that the convergence is uniform on compact sets. In the next subsection, we will show that the convergence happens also at a finer scale, ensuring convergence of discrete spatial derivatives to the derivatives of the limit. 

In this subsection and in all subsequent subsections, we will use the following conventions that were introduced in Subsection \ref{mainresult}. For $t\in \rr$, let $[t]$ denote the greatest integer $\le t$. For a point $x = (x_1,\ldots,x_d)\in \rr^d$, let $[x]$ denote the vector $([x_1],\ldots,[x_d])$. For $t\in \rr_{\ge 0}$, $x\in \rr^d$, and $\ve>0$, let
\[
t_\ve := [\ve^{-2} t], \ \ \ x_\ve := [\ve^{-1} x].
\] 
Since we are working under the assumption that $\phi(0)=0$, we have $f^{(\ve)}(t,x) = f_\ve(t_\ve, x_\ve)$. 
The key idea, unsurprisingly, is to show that the family $\{f^{(\ve)}\}_{\ve >0}$ is equicontinuous. This consists of the following two lemmas. 
\begin{lmm}\label{contlmm1}
Take any $t\in \rr_{\ge0}$ and $x,y\in\rr^d$. Then
\begin{align*}
\limsup_{\ve \to 0} |f^{(\ve)}(0,x) - f^{(\ve)}(t,y)| \le C|x-y| + Ct.
\end{align*}
\end{lmm}
\begin{proof}
By Lemma \ref{lipschitzlmm},
\begin{align}
|f^{(\ve)}(0,x)-f^{(\ve)}(0,y)|  &= |f_\ve(0, x_\ve) - f_\ve(0, y_\ve)| \notag\\
&\le C \ve | x_\ve -  y_\ve|\notag\\
&\le C|x-y|. \label{cont0}
\end{align}
On the other hand, by Proposition \ref{bdrycont},
\begin{align}
|f^{(\ve)}(0,y) - f^{(\ve)}(t,y)| &= |f_\ve(0, y_\ve) - f_\ve(t_\ve, y_\ve)|\notag\\
&=  |g_\ve(y_\ve) - f_\ve(t_\ve, y_\ve)|\notag\\
&\le C\ve \sqrt{t_\ve} + C\ve^2 t_\ve\le C t.\label{cont1}
\end{align}
The proof is completed by adding the two inequalities. 
\end{proof}
\begin{lmm}\label{contlmm2}
Take any real numbers $s> 0$ and $t\in [s,2s)$, and any $x,y\in \rr^d$. Then
\begin{align*}
&\limsup_{\ve \to 0} |f^{(\ve)}(s,y) - f^{(\ve)}(t,x)| \\
&\le C|x-y| + C(t-s) \biggl(1+\log \frac{s}{t-s}+ \frac{1+|x|+\sqrt{s}}{s}\biggr),
\end{align*}
where the second term is interpreted as zero if $s=t$.
\end{lmm}
\begin{proof}
By Lemma \ref{lipschitzlmm},
\begin{align}
|f^{(\ve)}(s,y)-f^{(\ve)}(s,x)|  &= |f_\ve(s_\ve, y_\ve) - f_\ve(s_\ve, x_\ve)| \notag\\
&\le C \ve | y_\ve - x_\ve|\notag \\
&\le C|x-y|. \label{conteq1}
\end{align}
If $s=t$, this completes the proof. So let us assume that $s<t$.  Let 
\[
r_\ve := t_\ve - s_\ve. 
\]
Since $0< t-s < s$, $r_\ve$ must be less than $s_\ve$ and bigger than $0$ when $\ve$ is small enough. Therefore by Proposition \ref{timeprop1}, 
\begin{align}
&|f^{(\ve)}(s,x)-f^{(\ve)}(t,x)| = |\delta_{r_\ve} f_\ve(s_\ve, x_\ve)|\notag \\
&\le C_1\ve^2 r_\ve(1+\log(s_\ve/r_\ve)) + C_1\ve^2(\log t_\ve)^{C_2} \notag \\
&\qquad + C_1 r_\ve s_\ve^{-1} (1+\ve |x_\ve|+\ve s_\ve^{1/2}) \notag \\
&\qquad + C_1 (1+\ve |x_\ve| + \ve s_\ve^{1/2}\log s_\ve) s_\ve^{-1}(\log s_\ve)^{C_2}.\label{conteq2}
\end{align}
Note that $\ve^2 r_\ve \to t-s$, $\ve^2 s_\ve \to s$, and $\ve |x_\ve|\to |x|$ as $\ve \to 0$. From this, it is easy to compute the limit of the right side in the above display as $\ve \to 0$. This completes the proof. 
\end{proof}
We are now ready to prove the main result of this subsection. 
\begin{prop}\label{subseqprop}
Given any sequence $\ve_n\to 0$, there is a subsequence along which $f^{(\ve)}$ converges pointwise everywhere on $\rr_{\ge0} \times \rr^d$. Moreover, the convergence is guaranteed to be uniform on compact subsets of $\rr_{\ge 0}\times \rr^d$. Any such limit $f$ is continuous and satisfies $f(0,x)=g(x)$ for all $x$. 
\end{prop}
\begin{proof}
Take any $t\in \rr_{\ge 0}$ and $x\in \rr^d$. By Proposition \ref{bdrycont} and the Lipschitz property of $g$, 
\begin{align}
|f^{(\ve)}(t,x)| &= |f_\ve(t_\ve, x_\ve)| \notag \\
&\le |g_\ve(x_\ve)|+ C\ve \sqrt{t_\ve} + C\ve^2 t_\ve\notag \\
&\le C(1+|x|+\sqrt{t} + t). \label{fenk}
\end{align}
In particular, $|f^{(\ve)}(t,x)|$ is bounded by a number that does not depend on $\ve$. Thus, given any sequence $\ve_n\to 0$, we can find (by a diagonal argument) a subsequence $\ve_{n_k}$ along which $f^{(\ve_{n_k})}(t,x)$ converges to a limit $f(t,x)$ for every $t\in \mathbb{Q}_{\ge 0}$ and $x\in \mathbb{Q}^d$, where $\mathbb{Q}$ is the set of rational numbers. 

Now take any $x\in  \rr^d$ and any $\delta>0$. Find $y\in \mathbb{Q}^d$ such that $|x-y| <\delta$. Then by Lemma \ref{contlmm1},
\begin{align*}
\limsup_{k\to \infty} |f^{(\ve_{n_k})}(0,x)-f(0,y)| \le C\delta. 
\end{align*}
In particular,
\[
\limsup_{k\to\infty} f^{(\ve_{n_k})}(0,x) - \liminf_{k\to\infty} f^{(\ve_{n_k})}(0,x) \le C\delta.
\]
Since this holds for any $\delta$, we see that $f^{(\ve_{n_k})}(0,x)$ converges to a limit as $k\to\infty$. Moreover, by \eqref{fenk}, the sequence is bounded. Hence the limit is finite. Let us call it $f(0,x)$. 

Next, take any $t >0$ and $x\in \rr^d$. Take any $s\in (t/2,t]\cap \mathbb{Q}$ and $y\in \mathbb{Q}^d$. Then by Lemma \ref{contlmm2},
\begin{align*}
&\limsup_{k\to \infty} |f(s,y) - f^{(\ve_{n_k})}(t,x)| \\
&\le C|x-y| + C(t-s) \biggl(1+\log \frac{s}{t-s}+ \frac{1+|x|+\sqrt{s}}{s}\biggr).
\end{align*}
The right side can be made arbitrarily small by bringing $y$ close to $x$ and $s$ close to $t$. Then by the same argument as before, we conclude that 
\[
f(t,x):=\lim_{k\to\infty} f^{(\ve_{n_k})}(t,x)
\]
exists and is finite.  Continuity of $f$ follows from Lemmas \ref{contlmm1} and \ref{contlmm2}. Proposition \ref{bdrycont} shows that $f(0,x)=g(x)$. Lastly, it is easy to deduce using the inequalities \eqref{conteq1} and \eqref{conteq2} from the proof of Lemma \ref{contlmm2} that if $(t_{n_k}, x_{n_k})\to (t,x)\in \rr_{>0}\times \rr^d$, then $f^{(\ve_{n_k})}(t_{n_k},x_{n_k}) \to f(t,x)$. If $(t_{n_k}, x_{n_k})\to (0,x)\in \{0\}\times \rr^d$, the same deduction can be made using the inequalities \eqref{cont0} and \eqref{cont1} from the proof of Lemma \ref{contlmm1}. It is a standard exercise to deduce uniform convergence on compact sets from these facts.
\end{proof}

\subsection{Differentiability of subsequential limits}\label{diffsec}
The goal of this subsection is to show that if $f^{(\ve)}$ converges to $f$ along a subsequence, then $f$ is differentiable in $x$ (for any fixed $t$), and the first-order discrete derivatives of $f^{(\ve)}$ converge pointwise to the corresponding derivatives of $f$. Moreover, the derivatives are continuous in $t$ and $x$. As in the previous subsection, the results of this subsection are also based on the estimates from Subsection \ref{surfacederivsec}. 

Let $e_1,\ldots, e_d$ be the standard basis vectors of $\rr^d$. For $t\in \zz_{\ge0}$ and $x\in \zz^d$, define the discrete derivative 
\[
D_i f_\ve(t,x) := \frac{f_\ve(t,x+e_i)-f_\ve(t, x)}{\ve}. 
\]
Also, for $t\in \rr_{\ge0}$ and $x\in \rr^d$, define
\[
D_i f^{(\ve)}(t,x) := D_i f_\ve(t_\ve, x_\ve). 
\]
Let $Df^{(\ve)}(t,x)$ denote the vector $(D_i f^{(\ve)}(t,x))_{1\le i\le d}$. We need several lemmas about the behaviors of these discrete derivatives. The first lemma shows that they are uniformly bounded. 
\begin{lmm}\label{lip3lmm}
For each $1\le i\le d$,  $|D_if_\ve|$ and $|D_i f^{(\ve)}|$ are bounded by $L$ everywhere.
\end{lmm}
\begin{proof}
This is just a restatement of Lemma \ref{lipschitzlmm}. 
\end{proof}
The next lemma shows that the discrete derivatives have a certain degree of smoothness. 
\begin{lmm}\label{discderiv}
For any $1\le i, j\le d$, $t\in \zz_{\ge 1}$, $x\in \zz^d$, and $k\in \zz\setminus\{0\}$, we have 
\begin{align*}
|D_i f_\ve(t, x + ke_j) - D_if_\ve(t,x)|  \le J(\ve, k,t,x),
\end{align*}
where
\begin{align*}
&J(\ve, k,t,x) \\
&:= C_1 \ve |k|\log (2+t/k^2)  \notag\\
&\qquad + C_1\ve^{-1}(1+\ve|x| +\ve|k|+ \ve t^{1/2})t^{-1/2}\min\{1, t^{-1/2}|k|\} \notag\\
&\qquad + C_1\ve^{-1} (1+\ve |x| + \ve|k| + \ve t^{1/2}\log t) |k|t^{-2}(\log t)^{C_2}
\end{align*}
for some suitable constants $C_1$ and $C_2$ that depend only on $\phi$, $L$ and $d$.
\end{lmm}
\begin{proof}
Note that
\begin{align*}
D_i f_\ve(t, x + ke_j) - D_if_\ve(t,x) &= \ve^{-1}(\delta_{e_i} f_\ve(t, x+ke_j) - \delta_{e_i} f_\ve(t,x))\\
&= \ve^{-1} \delta_{e_i} \delta_{ke_j} f_\ve(t,x). 
\end{align*}
The desired bound now follows from Lemma \ref{secondderiv}. 
\end{proof}
The function $J$ defined in the above lemma describes an important remainder term. We need the following fact about it.
\begin{lmm}\label{jlemma}
For any positive integers $k$ and $t$, any $\ve\in (0,1)$, and any $x\in \zz^d$,
\begin{align*}
\sum_{j=1}^k J(\ve,j,t,x) &\le kJ(\ve, k,t,x) + C\ve k^2. 
\end{align*}
\end{lmm}
\begin{proof}
Let $C_1$ be the constant displayed in the definition of $J$. From the definition of $J$, it is immediate that for any $\ve$, $t$ and $x$, the function
\[
J(\ve, k, t,x) + 2C_1\ve|k|\log |k|
\]
is increasing in $|k|$. Thus, for any positive integers $k\ge j\ge 1$, 
\[
J(\ve, j,t,x) \le J(\ve, k, t,x) + 2C_1\ve k\log(k/j).
\]
The proof is completed by observing that $\sum_{j=1}^k\log(k/j) \le Ck$ for some universal constant $C$. 
\end{proof}
\begin{lmm}\label{discderiv2}
For any $1\le i\le d$, $t\in \zz_{\ge 1}$, $x\in \zz^d$, and $k\in \zz\setminus\{0\}$, 
\begin{align*}
|f_\ve(t, x + ke_i) - f_\ve(t,x) - k\ve D_if_\ve(t,x)|  \le \ve |k| J(\ve, k,t,x) + C\ve^2k^2. 
\end{align*}
\end{lmm}
\begin{proof}
Suppose that $k>0$. Then note that 
\begin{align*}
&f_\ve(t, x + ke_i) - f_\ve(t,x) - k\ve D_i f_\ve(t,x) \\
&=\ve \sum_{j=1}^{k-1} (D_i f_\ve(t, x+je_i) - D_i f_\ve(t,x)). 
\end{align*}
Bounding the absolute value of each term in the above sum using Lemma \ref{discderiv}, we get 
\begin{align*}
|f_\ve(t, x + ke_i) - f_\ve(t,x) - k\ve D_i f_\ve(t,x)| &\le \ve\sum_{j=1}^{k-1} J(\ve, j, t,x). 
\end{align*}
By Lemma \ref{jlemma}, this completes the proof when $k>0$. The proof for $k<0$ is similar, since Lemma \ref{discderiv} works for both positive and negative $k$ in the same way. 
\end{proof}
\begin{lmm}\label{discderiv3}
For any $1\le i\le d$, $t\in \rr_{>0}$, $x\in \rr^d$, and $a \in \rr$,
\begin{align*}
&\limsup_{\ve \to 0} \biggl|\frac{f^{(\ve)}(t, x+a e_i) - f^{(\ve)}(t,x)}{a} - D_i f^{(\ve)}(t,x)\biggr| \le Q(a,t,x),
\end{align*}
where
\begin{align*}
Q(a,t,x) &:= C|a|\log (2+t/a^2)  \\
&\qquad + C(1+|x| +|a|+ t^{1/2})t^{-1/2}\min\{1, t^{-1/2}|a|\}
\end{align*}
for some suitable constant $C$ that depends only on $\phi$, $L$ and $d$.
\end{lmm}
\begin{proof}
Let $k_\ve$ be the integer such that 
\begin{align}\label{integereq}
[\ve^{-1} (x + a e_i)] = x_\ve + k_\ve e_i.
\end{align}
(Note that there is such an integer because $[\ve^{-1}(x+a e_i)]$ and $x_\ve$ agree on  all coordinates other than $i$.) Note also that
\begin{align*}
&f^{(\ve)}(t, x+a e_i) - f^{(\ve)}(t,x) - a D_i f^{(\ve)}(t,x) \\
&= f_\ve(t_\ve, x_\ve + k_\ve e_i) - f_\ve(t_\ve, x_\ve) - a D_i f_\ve (t_\ve,x_\ve).
\end{align*}
Let $p$ be the value of coordinate $i$ of the vector $x_\ve$, and let $q$ be the value of coordinate $i$ of the vector $[\ve^{-1}(x+a e_i)]$. Then on the one hand, $|q - \ve^{-1} (x_i + a)| \le 1$. On the other hand, by \eqref{integereq}, $q = p + k_\ve$. But $|p-\ve^{-1}x_i|\le 1$. Combining these three observations, we get 
\begin{align}\label{keq}
|k_\ve -\ve^{-1}a|\le 2.
\end{align}
Combining this with Lemma \ref{lip3lmm}, we get
\[
\lim_{\ve \to 0} (a-k_\ve \ve) D_i f_\ve(t_\ve, x_\ve) = 0. 
\]
Thus,
\begin{align*}
&\limsup_{\ve \to 0} |f^{(\ve)}(t, x+a e_i) - f^{(\ve)}(t,x) - a D_i f^{(\ve)}(t,x) |\\
&= \limsup_{\ve \to 0} |f_\ve(t_\ve, x_\ve + k_\ve e_i) - f_\ve(t_\ve, x_\ve) - k_\ve \ve D_i f_\ve (t_\ve,x_\ve)|.
\end{align*}
By Lemma \ref{discderiv2},
\begin{align*}
&|f_\ve(t_\ve, x_\ve + k_\ve e_i) - f_\ve(t_\ve,x_\ve) - k_\ve\ve D_if_\ve(t_\ve,x_\ve)| \\
&\le \ve |k_\ve| J(\ve, k_\ve, t_\ve, x_\ve) + C\ve^2k_\ve^2.
\end{align*}
Note that as $\ve \to 0$, $\ve x_\ve \to x$ and $\ve^2 t_\ve \to t$. Also, by \eqref{keq}, $\ve k_\ve \to a$. Plugging these limits into the right side of the above display, we get the required result.
\end{proof}

Let us now fix a sequence $\ve_n \to0$ such that $f^{(\ve_n)}$ converges pointwise (and uniformly on compact sets) to a limit $f$. Such a sequence exists by Proposition~\ref{subseqprop}. Moreover, the function $f$ is continuous and satisfies $f(0,x)=g(x)$ for all $x$.  We will now show that $f$ is differentiable in $x$, the derivatives are continuous in $x$ and $t$, and the discrete first-order derivatives of $f^{(\ve_n)}$ converge to the corresponding derivatives of $f$. In the next two lemmas, $Q$ denotes the function appearing as the error bound in Lemma \ref{discderiv3}.
\begin{lmm}\label{discderiv4}
For any $1\le i\le d$, $t\in \rr_{>0}$, $x\in \rr^d$, and $a,b \in \rr\setminus \{0\}$,
\begin{align*}
\biggl|\frac{f(t, x+a e_i) - f(t,x)}{a} -\frac{f(t, x+b e_i) - f(t,x)}{b}\biggr| &\le Q(a,t,x) + Q(b,t,x).
\end{align*}
\end{lmm}
\begin{proof}
Applying Lemma \ref{discderiv3} and dividing both sides by $a$, and then doing  the same with $a$ replaced by $b$, and finally, applying the triangle inequality, we arrive at 
\begin{align*}
&\limsup_{\ve \to 0} \biggl|\frac{f^{(\ve)}(t, x+a e_i) - f^{(\ve)}(t,x)}{a} -\frac{f^{(\ve)}(t, x+b e_i) - f^{(\ve)}(t,x)}{b}\biggr| \\
&\le Q(a,t,x) + Q(b,t,x).
\end{align*}
The proof is now completed by replacing $\ve$ by $\ve_n$ and $\limsup_{\ve \to 0}$ by $\lim_{n\to\infty}$.
\end{proof}

Note that $Q(a,t,x) \to 0$ as $a \to 0$. Therefore, Lemma \ref{discderiv4} implies that as $a\to 0$, the numbers $(f(t, x+a e_i) - f(t,x))/a$ have the Cauchy property, and hence the limit
\[
 \lim_{a\to 0} \frac{f(t, x+a e_i) - f(t,x)}{a}
\] 
exists and is finite. Let us call this limit $\partial_i f(t,x)$, and let us denote the vector $(\partial_i f(t,x))_{1\le i\le d}$ by $\nabla f(t,x)$. We will show below that $f$ is differentiable in $x$ with gradient $\nabla f$, and also that $D f^{(\ve_n)}$ converges to $\nabla f$. The following lemma is the key step in the proof.
\begin{lmm}\label{difflmm}
For any $t\in \rr_{>0}$ and $x\in \rr^d$, and any $1\le i\le d$,
\begin{align*}
\lim_{n\to\infty} D_i f^{(\ve_n)}(t,x) = \partial_i f(t,x). 
\end{align*}
Moreover, for any $a\ne 0$, 
\[
 \biggl|\frac{f(t, x+a e_i) - f(t,x)}{a} -  \partial_i f(t,x)\biggr| \le Q(a,t,x). 
\]
\end{lmm}
\begin{proof}
By Lemma \ref{discderiv3}, for any $a\ne 0$, 
\begin{align}\label{limftx}
&\limsup_{n \to \infty} \biggl|\frac{f(t, x+a e_i) - f(t,x)}{a} -  D_i f^{(\ve_n)}(t,x)\biggr| \le Q(a,t,x). 
\end{align}
Consequently,
\begin{align*}
\frac{f(t, x+a e_i) - f(t,x)}{a} - Q(a,t,x) &\le \liminf_{n\to\infty}  D_i f^{(\ve_n)}(t,x)\\
&\le   \limsup_{n\to\infty}  D_i f^{(\ve_n)}(t,x)\\
&\le \frac{f(t, x+a e_i) - f(t,x)}{a} + Q(a,t,x).
\end{align*}
Since $\partial_i f(t,x) = \lim_{a\to0} (f(t,x+a e_i)-f(t,x))/a$ and $Q(a,t,x)\to0$ as $a\to 0$, this proves the first claim.  The second claim follows by combining the first claim with the  inequality \eqref{limftx}.
\end{proof} 
We now arrive at the main result of this subsection.
\begin{prop}\label{diffcont}
The function $\nabla f : \rr_{>0}\times \rr^d \to \rr^d$ is continuous. Moreover, the function $Df^{(\ve_n)}$ converges to $\nabla f$ uniformly on any compact subset of $\rr_{>0}\times \rr^d$ as $n\to\infty$. In particular, for any sequence of points  $(t_n,x_n)$ converging to a point $(t,x)\in \rr_{>0}\times \rr^d$, we have $Df^{(\ve_n)}(t_n,x_n) \to \nabla f(t,x)$.
\end{prop}
\begin{proof}
Fix $t>0$ and $1\le i\le d$. Take any $x\in \rr^d$, $z\in \rr^d\setminus\{0\}$, and $\ve>0$. Let 
\[
z_\ve' := (x+z)_\ve - x_\ve.
\]
Then note that 
\begin{align*}
D_i f^{(\ve)}(t,x+z) - D_i f^{(\ve)}(t,x) &= \ve^{-1}(\delta_{e_i} f_\ve(t_\ve, x_\ve+z_\ve') - \delta_{e_i} f_\ve(t_\ve, x_\ve))\\
&= \ve^{-1}\delta_{e_i}\delta_{z_\ve'} f(t_\ve,x_\ve).
\end{align*}
Therefore, by Lemma \ref{secondderiv},
\begin{align}
&|D_i f^{(\ve)}(t,x+z) - D_i f^{(\ve)}(t,x)| \notag\\
&\le C_1 \ve |z_\ve'|\log (2+t_\ve/|z_\ve'|^2) \notag \\
&\qquad + C_1\ve^{-1}(1+\ve|x_\ve| +\ve|z_\ve'|+\ve t_\ve^{1/2})t_\ve^{-1/2}\min\{1, t_\ve^{-1/2}|z_\ve'|\} \notag\\
&\qquad + C_1\ve^{-1}(1+\ve |x_\ve| + \ve|z_\ve'|+\ve t_\ve^{1/2}\log t_\ve) |z_\ve'| t_\ve^{-2}(\log t_\ve)^{C_2}.\label{compact1}
\end{align}
Now note that as $\ve \to0$,  $\ve^2t_\ve \to t$, $\ve x_\ve \to x$ and $\ve z_\ve' \to z$. Thus, replacing $\ve$ by $\ve_n$ in the above display, taking $n\to\infty$, and using Lemma \ref{difflmm}, we get
\begin{align*}
&|\partial_i f(t,x+z)-\partial_i f(t,x)| \\
&\le C |z|\log (2+t/|z|^2) + C(1+|x| +|z|+t^{1/2})t^{-1/2}\min\{1, t^{-1/2}|z|\}.
\end{align*}
Since the right side tends to zero as $z\to0$, this proves the continuity of $\partial_i f(t,x)$ in $x$. Next, take some $r\in (0,t)$, and let 
\[
r_\ve' := (t+r)_\ve - t_\ve.
\]
Then   
\begin{align*}
D_i f^{(\ve)}(t+r,x) - D_i f^{(\ve)}(t,x) &= \ve^{-1}(\delta_{e_i} f_\ve(t_\ve+ r_\ve', x_\ve) - \delta_{e_i} f_\ve(t_\ve, x_\ve))\\
&= \ve^{-1}\delta_{r_\ve'}'\delta_{e_i}f(t_\ve,x_\ve).
\end{align*}
Therefore, by Proposition \ref{timeprop2},
\begin{align}
&|D_i f^{(\ve)}(t+r,x) - D_i f^{(\ve)}(t,x)|\notag\\
&\le C_1 \ve \sqrt{r_\ve'} + C_1 \ve  {r_\ve'}^{-1/2}(\log t_\ve)^{C_2}\notag \\
&\qquad + C_1 \ve^{-1}r_\ve' t_\ve^{-3/2} (1+\ve |x_\ve |+\ve + \ve t_\ve^{1/2}) \notag \\
&\qquad + C_1\ve^{-1}(1+\ve |x_\ve| + \ve + \ve t_\ve^{1/2}\log t_\ve) t_\ve^{-3/2}(\log t_\ve)^{C_2}.\label{compact2}
\end{align}
Replacing $\ve$ by $\ve_n$ and letting $n\to\infty$, we get
\begin{align*}
|\partial_i f(t+r,x) - \partial_i f(t,x)| &\le C\sqrt{r} + Crt^{-3/2}(1+|x|+\sqrt{t}). 
\end{align*}
This proves the continuity of $\partial_i f$ in $t$. 

Next, by the bounds \eqref{compact1} and \eqref{compact2} and the first assertion of Lemma \ref{difflmm}, it is easy to see that if $(t_n,x_n)$ is a sequence of points converging to a point $(t,x)\in \rr_{>0}\times \rr^d$, then $Df^{(\ve_n)}(t_n,x_n)\to \nabla f(t,x)$. It is a standard fact that this is equivalent to uniform convergence on compact sets.
\end{proof}
The following result is an immediate corollary of Proposition \ref{diffcont}.
\begin{cor}\label{diffprop}
For any $t>0$, $f(t,x)$ is continuously differentiable in $x$, with derivative $\nabla f$ defined above. Moreover, $|\nabla f|$ is uniformly bounded by $L\sqrt{d}$. 
\end{cor}
\begin{proof}
By definition, $\partial_i f$ is the partial derivative of $f$ in direction $i$. Since these partial derivatives are continuous by Proposition \ref{diffcont}, it follows that $f(t,\cdot)$ is continuously differentiable for any fixed $t$, and  $\nabla f$ is its gradient. Lemma \ref{lip3lmm} shows that $|D f^{(\ve)}|$ is uniformly bounded by $L\sqrt{d}$ everywhere, for any $\ve$. By Lemma \ref{difflmm}, this shows that $|\nabla f|$ is also bounded by the same quantity. 
\end{proof}

\subsection{Origin of the gradient squared field}\label{gradsec}
As in the previous subsection, let us fix a sequence $\ve_n\to0$ such that $f^{(\ve_n)}$ converges to a limit $f$. 
In this subsection, we will show that a suitably rescaled version of $h_{\ve_n}$ converges to a scalar multiple of $|\nabla f|^2$. Using Proposition \ref{fformlmm}, this will allow us later to derive an integral equation for $f$.

Take any $b, b'\in B$ such that $b\ne b'$ and $b\ne -b'$. Let
\[
\gamma_1 := \partial_b^2 \phi(0), \ \ \gamma_2 := \partial_b\partial_{-b} \phi(0), \ \  \gamma_3 := \partial_b\partial_{b'} \phi(0).
\]
By the invariance of $\phi$ under lattice symmetries, these numbers do not depend on the choices of $b$ and $b'$. Note that the quantity $\gamma$ in the statement of Theorem \ref{kpzunivthm} equals $\gamma_1-\gamma_2$. Define
\begin{align*}
H_\ve(t,x) &:= \frac{\gamma_1}{2}\sum_{b\in B} (\delta_b f_\ve(t-1,x))^2 + \frac{\gamma_2}{2}\sum_{b\in B} \delta_b f_\ve(t-1,x)\delta_{-b}f_\ve(t-1,x) \\
&\qquad + \frac{\gamma_3}{2}\sum_{\substack{b,b'\in B \\ b\ne b', b\ne -b'}} \delta_b f_\ve(t-1,x)\delta_{b'} f_\ve(t-1,x). 
\end{align*}
The following lemma shows that $h_\ve$ is close to $H_\ve$. 
\begin{lmm}\label{hehe}
For any $\ve>0$, $t\in \zz_{\ge 1}$ and $x\in \zz^d$,
\begin{align*}
|h_\ve(t,x) - H_\ve(t,x)| &\le \ve^2F(\ve),
\end{align*}
where $F$ is a function determined solely by $\phi$ (and not depending on $t$ or $x$), such that $F(\ve)\to 0$ as $\ve \to 0$. 
\end{lmm}
\begin{proof}
Let $u\in \rr^A$ be the vector whose coordinates are $u_0=0$ and 
\[
u_b = \delta_b f_\ve (t-1,x)
\]
for $b\in B$. Then recall the identity \eqref{ftx}, which says that 
\[
f_\ve(t,x) = f_\ve(t-1,x) + \phi(u).
\]
By Lemma \ref{lipschitzlmm}, $|u_b|\le C\ve$ for all $b$. Since $\phi$ is a $C^2$ function and $\phi(0)=0$, this implies that
\begin{align*}
&\biggl|f_\ve(t,x) - f_\ve(t-1,x) - \sum_{b\in B} u_b\partial_b \phi(0) - \frac{1}{2}\sum_{b,b'\in B} u_bu_{b'}\partial_b\partial_{b'} \phi(0)\biggr| \\
&\le C\ve^2 F(\ve),
\end{align*} 
where $F$ is determined solely by $\phi$, and $F(\ve)\to 0$ as $\ve \to 0$. It is easy to check that the quantity on the left equals $|h_\ve(t,x)-H_\ve(t,x)|$. 
\end{proof}

For $t\in \rr_{>0}$ and $x\in \rr^d$, define
\begin{align*}
h^{(\ve)}(t,x) &:= \ve^{-2} h_\ve(t_\ve, x_\ve),\\
H^{(\ve)}(t,x) &:= \ve^{-2} H_\ve(t_\ve, x_\ve). 
\end{align*}
The following proposition is the main result of this subsection. It shows that $h^{(\ve_n)}$ converges pointwise to $|\gamma|\nabla f|^2$. The first step in the proof is to show that this holds for $H^{(\ve_n)}$, and then use Lemma \ref{hehe} to deduce that it also holds for $h^{(\ve_n)}$. 
\begin{prop}\label{hepprop}
For any sequence $(t_n,x_n)$ converging to a point $(t,x)\in \rr_{>0}\times \rr^d$, we have
\[
\lim_{n\to\infty} h^{(\ve_n)}(t_n,x_n) = \gamma|\nabla f(t,x)|^2.
\]
\end{prop}
\begin{proof}
For any $t\in \zz_{\ge 1}$ and $x\in \zz^d$, note that
\begin{align*}
\delta_{e_i} f_\ve(t-1,x) &= \ve D_i f_\ve(t-1,x)\\
&= \ve D_if^{(\ve)}(\ve^{2}(t-1), \ve  x).
\end{align*}
Similarly, 
\begin{align*}
\delta_{-e_i} f_\ve(t-1,x) &= -\ve D_i f_\ve(t-1, x-e_i)\\
&= -\ve D_i f^{(\ve)}(\ve^{2}(t-1), \ve(x-e_i)). 
\end{align*}
Next, changing notation, let $(t_n,x_n$) and $(t,x)$ be as in the statement of the proposition. Let $s_n := [\ve_n^{-2} t_n]$ and $y_n := [\ve_n^{-1}x_n]$. Then by the above identities,
\begin{align*}
&H^{(\ve_n)} (t_n,x_n) = \ve_n^{-2}H_{\ve_n}(s_n, y_n)\\
&= \frac{\gamma_1}{2}\sum_{i=1}^d \biggl[(D_i f^{(\ve_n)} (\ve_n^2(s_n-1), \ve_n y_n))^2 \\
&\qquad \qquad + (D_i f^{(\ve_n)}(\ve_n^2(s_n-1), \ve_n (y_n-e_i)))^2\biggr] \\
&\quad - \gamma_2 \sum_{i=1}^d D_i f^{(\ve_n)} (\ve_n^2(s_n-1), \ve_n y_n) D_i f^{(\ve_n)} (\ve_n^2(s_n-1), \ve_n (y_n-e_i))\\
&\quad + \gamma_3 \sum_{1\le i\ne j\le d} D_i f^{(\ve_n)} (\ve_n^2(s_n-1), \ve_n y_n) D_j f^{(\ve_n)} (\ve_n^2(s_n-1), \ve_n y_n)\\
&\quad - \gamma_3 \sum_{1\le i\ne j\le d} D_i f^{(\ve_n)} (\ve_n^2(s_n-1), \ve_n y_n) D_j f^{(\ve_n)} (\ve_n^2(s_n-1), \ve_n (y_n-e_j)).
\end{align*}
Now note that $\ve_n^2s_n \to t$ and $\ve_n y_n \to x$ as $n \to \infty$. Therefore by Proposition~\ref{diffcont}, we have the following limits, for each $i$:
\begin{align*}
&\lim_{n\to\infty} D_i f^{(\ve_n)} (\ve_n^2(s_n-1), \ve_n y_n) = \partial_i f(t,x),\\
&\lim_{n\to\infty} D_i f^{(\ve_n)} (\ve_n^2(s_n-1), \ve_n (y_n-e_i)) = \partial_i f(t,x). 
\end{align*}
Plugging these into the previous display, we get
\[
\lim_{n\to\infty} H^{(\ve_n)}(t_n,x_n) = (\gamma_1-\gamma_2)|\nabla f(t,x)|^2 = \gamma|\nabla f(t,x)|^2. 
\]
But by Lemma \ref{hehe}, the limit of $h^{(\ve_n)}(t_n,x_n)$ must also be the same. 
\end{proof}

\subsection{Duhamel representation for subsequential limits}\label{duhamelsec}
As in the previous two subsections, fix a sequence $\ve_n\to0$ such that $f^{(\ve_n)}$ converges to a limit $f$. In this subsection we will derive an integral equation for $f$. The idea is to use Proposition \ref{hepprop} to show that the recursive equation displayed in Proposition~\ref{fformlmm} yields an integral equation for $f$ in the scaling limit.

Recall the transition probability $p(t,x)$ defined in Subsection \ref{rwsec}. For $(t,x)\in \rr_{\ge0}\times \rr^d$ and $\ve>0$, define
\begin{align*}
p^{(\ve)}(t,x) := \ve^{-d}p(t_\ve, x_\ve).
\end{align*}
Also, recall the Gaussian kernel $K(t,x) = (4\pi \beta t)^{-d/2} e^{-|x|^2/4\beta t}$ defined in the statement of Theorem \ref{kpzunivthm}.  The following result shows that $p^{(\ve)}$ converges to $K$ uniformly on compact sets as $\ve \to 0$.
\begin{lmm}\label{plimlmm}
For any sequence $(t_n,x_n)$ converging to a point $(t,x)\in \rr_{> 0}\times \rr^d$ and any sequence $\ve_n\to 0$, 
\[
\lim_{n \to \infty} p^{(\ve_n)}(t_n,x_n) = K(t,x).
\]
\end{lmm}
\begin{proof}
Recall the functions $p_1$, $p_2$, and $p_3$ from Subsection \ref{approxsec}. For $(t,x)\in \rr_{>0} \times \rr^d$, define
\[
p_3^{(\ve)}(t,x) := \ve^{-d} p_3(t_\ve, x_\ve). 
\]
Note that $(t_n)_{\ve_n}\sim \ve_n^{-2} t$ as $n\to\infty$, since $t_n\to t>0$. Thus, by Lemmas \ref{pp1lmm}, \ref{p2p3lmm} and \ref{p1p2lmm}, we have that as $n\to\infty$,
\begin{align*}
p^{(\ve_n)}(t_n,x_n) - p_1^{(\ve_n)}(t_n,x_n) = O(\ve_n^{-d}e^{-C(\log \ve_n)^2}),\\
p_1^{(\ve_n)}(t_n,x_n) - p_2^{(\ve_n)}(t_n,x_n) = O(\ve_n^{-d}\ve_n^{d+2}|\log \ve_n|^C),\\
p_2^{(\ve_n)}(t_n,x_n) - p_3^{(\ve_n)}(t_n,x_n) = O(\ve_n^{-d}e^{-C(\log \ve_n)^2}).
\end{align*}
All of the above tend to zero as $n\to\infty$. Thus, 
\begin{align*}
\lim_{n\to \infty} (p^{(\ve_n)}(t_n,x_n) - p_3^{(\ve_n)}(t_n,x_n)) = 0. 
\end{align*}
On the other hand, from the explicit expression \eqref{p3form} for $p_3$, we have that
\[
\lim_{n\to \infty} p_3^{(\ve_n)}(t_n,x_n) = K(t,x). 
\]
This completes the proof of the lemma.
\end{proof}

Define $g^{(\ve)} : \rr^d \to \rr$ as 
\[
g^{(\ve)}(x) := g_\ve(x_\ve) = g(\ve[\ve^{-1}x]). 
\]
Recall the definition of $h^{(\ve)}$ from the previous subsection. The following lemma expresses the recursive formula from Proposition \ref{fformlmm} as an integral involving the functions $p^{(\ve)}$, $g^{(\ve)}$ and $h^{(\ve)}$. 
\begin{lmm}\label{feform}
For any $(t,x)\in \rr_{>0}\times \rr^d$ and $\ve\in (0,\sqrt{t})$,
\begin{align*}
&f^{(\ve)}(t,x) \\
&= \int_{\rr^d} p^{(\ve)}(\ve^2t_\ve, \ve(x_\ve-y_\ve)) g^{(\ve)}(\ve y_\ve)dy \\
&\qquad + \int_0^{\ve^{2}t_\ve} \int_{\rr^d} p^{(\ve)}(\ve^2 s_\ve, \ve(x_\ve-y_\ve)) h^{(\ve)}(\ve^2(t_\ve-s_\ve), \ve y_\ve)dy ds.
\end{align*}
\end{lmm}
\begin{proof}
The map $x\mapsto [x]$ takes cubes of unit volume in $\rr^d$ to elements of $\zz^d$. The map is surjective, and these cubes cover the whole of $\rr^d$. This implies that for any absolutely summable $w:\zz^d\to \rr$,
\begin{align}\label{wxeq}
\sum_{x\in \zz^d} w(x) = \int_{\rr^d} w([x])dx.
\end{align}
Similarly, for any absolutely summable $w:\zz \to \rr$, and any two integers $a<b$,
\[
\sum_{t=a}^{b-1} w(t) = \int_a^b w([t])dt.
\]
The condition  $\ve <\sqrt{t}$ implies that $t_\ve \ge 1$. So, using the above identities and Proposition \ref{fformlmm}, we get
\begin{align*}
&f^{(\ve)}(t,x) = f_\ve(t_\ve,x_\ve) \\
&= \sum_{y\in \zz^d} p(t_\ve,x_\ve-y) g_\ve(y) + \sum_{0\le s\le t_\ve-1} \sum_{y\in \zz^d} p(s, x_\ve-y) h_\ve(t_\ve-s,y)\\
&= \int_{\rr^d} p(t_\ve,x_\ve-[y]) g_\ve([y])dy + \int_0^{t_\ve} \int_{\rr^d} p([s], x_\ve-[y]) h_\ve(t_\ve-[s],[y])dy ds.
\end{align*}
Applying the changes of variable $z = \ve y$ and $u = \ve^2 s$ in the above integrals,  we get 
\begin{align*}
&f^{(\ve)}(t,x) \\
&= \ve^{-d}\int_{\rr^d} p(t_\ve,x_\ve-z_\ve) g_\ve(z_\ve)dz \\
&\qquad + \ve^{-(d+2)}\int_0^{\ve^{2}t_\ve} \int_{\rr^d} p(u_\ve, x_\ve-z_\ve) h_\ve(t_\ve-u_\ve, z_\ve)dz du\\
&= \int_{\rr^d} p^{(\ve)}(\ve^2t_\ve, \ve(x_\ve-z_\ve)) g^{(\ve)}(\ve z_\ve)dz \\
&\qquad + \int_0^{\ve^{2}t_\ve} \int_{\rr^d} p^{(\ve)}(\ve^2 u_\ve, \ve(x_\ve-z_\ve)) h^{(\ve)}(\ve^2(t_\ve-u_\ve), \ve z_\ve)dz du.
\end{align*}
This completes the proof of the lemma.
\end{proof}
Our goal is to take $\ve$ to zero in the integral equation from Lemma \ref{feform}. For that, we need to apply the dominated convergence theorem. The next two lemmas prepare the ground for that.
\begin{lmm}\label{domlmm}
Take any $\delta>0$. There are constants $C_1$ and $C_2$, depending only on $d$ and $\delta$, such that for any $t\in \rr_{>\delta}$, $x,y\in \rr^d$, and $\ve \in (0,\sqrt{\delta})$,
\[
p^{(\ve)}(\ve^2t_\ve,\ve(x_\ve-y_\ve)) \le C_1e^{-C_2|x-y|^2/t}.
\]
\end{lmm}
\begin{proof}
In this proof, $C_1, C_2,\ldots$ will denote constants that may depend only on $d$ and $\delta$. Note that $t_\ve \ge 1$. So, by Lemmas \ref{pp1lmm}, \ref{p2p3lmm} and \ref{p1p2lmm},
\begin{align}
&p^{(\ve)}(\ve^2t_\ve,\ve(x_\ve-y_\ve)) = \ve^{-d} p(t_\ve, x_\ve - y_\ve)\notag \\
&\le \ve^{-d} p_3(t_\ve, x_\ve - y_\ve) + C_1 \ve^{-d} t_\ve^{-(d+2)/2} (\log t_\ve)^{C_2}\notag \\
&= \ve^{-d}(4\pi\beta t_\ve)^{-d/2} e^{-|x_\ve-y_\ve|^2/4\beta t_\ve}+C_1 \ve^{-d} t_\ve^{-(d+2)/2} (\log t_\ve)^{C_2}. \label{pzero}
\end{align}
Now note that since $t> \delta > \ve^2$, we have 
\begin{align}\label{tineq}
\frac{1}{2}\ve^{-2}t \le t_\ve \le \ve^{-2} t.
\end{align}
Also, since $|\ve^{-1}x - x_\ve|\le \sqrt{d}$ and $|\ve^{-1}y-y_\ve|\le \sqrt{d}$, we have
\begin{align}
\ve^{-2}|x - y|^2 &\le 3(|\ve^{-1}x - x_\ve|^2 +  |x_\ve - y_\ve|^2 + |\ve^{-1}y-y_\ve|^2)\notag \\
&\le C_1 + C_2|x_\ve - y_\ve|^2. \label{xyineq}
\end{align}
Using \eqref{tineq} and \eqref{xyineq} in \eqref{pzero}, we get
\begin{align*}
&p^{(\ve)}(\ve^2t_\ve,\ve(x_\ve-y_\ve)) \\
&\le  C_1 t^{-d/2} e^{-C_2|x-y|^2/t} + C_1\ve^2 t^{-(d+2)/2}(\log (\ve^{-2} t))^{C_2}. 
\end{align*}
Similarly, by \eqref{tineq}, \eqref{xyineq}, and Lemma \ref{plmm},
\[
p^{(\ve)}(\ve^2t_\ve,\ve (x_\ve-y_\ve)) \le C_1 \ve^{-d}e^{-C_2|x-y|^2/t}. 
\]
Combining the last two displays, we get
\begin{align}
&p^{(\ve)}(\ve^2t_\ve,\ve (x_\ve-y_\ve))  \notag\\
&\le C_1 t^{-d/2} e^{-C_2|x-y|^2/t}\notag\\
&\qquad + C_1 \min\{\ve^2 t^{-(d+2)/2}(\log (\ve^{-2} t))^{C_2}, \ve^{-d}e^{-C_3|x-y|^2/t}\}. \label{minbd}
\end{align}
Using the inequality $\min\{a,b\}\le a^{(d+1)/(d+2)} b^{1/(d+2)}$, we get
\begin{align}
 &\min\{\ve^2 t^{-(d+2)/2}(\log (\ve^{-2}t))^{C_2}, \ve^{-d}e^{-C_3|x-y|^2/t}\} \notag \\
 &\le\ve  t^{-(d+1)/2} e^{-C_4|x-y|^2/t} (\log (\ve^{-2} t))^{C_5}\le C_6 e^{-C_7|x-y|^2/t},\label{minjust}
\end{align}
where the last inequality holds because 
\[
\ve t^{-(d+1)/2} (\log (\ve^{-2} t))^{C_5} \le \frac{C_8(\log (\ve^{-2} t))^{C_5}}{(\ve^{-2} t)^{1/2}} \le C_9. 
\]
Plugging the bound from \eqref{minjust} into the right side of \eqref{minbd} completes the proof of the lemma.
\end{proof}

\begin{lmm}\label{remlmm}
For any $t\in \rr_{\ge 0}$, $x\in \rr^d$, and $\ve \in (0,1)$,
\[
 \int_{\rr^d} p^{(\ve)}(\ve^2t_\ve, \ve(x_\ve-y_\ve))dy = 1.
\]
\end{lmm}
\begin{proof}
Note that by \eqref{wxeq},
\begin{align*}
 \int_{\rr^d} p^{(\ve)}(\ve^2t_\ve, \ve(x_\ve-y_\ve)) dy &= \ve^{-d}\int_{\rr^d} p(t_\ve, x_\ve - y_\ve) dy\\
 &= \int_{\rr^d} p(t_\ve, x_\ve - [z]) dz\\
 &= \sum_{z\in \zz^d} p(t_\ve, x_\ve - z) = 1,
\end{align*}
where the last equality holds because $p(t_\ve,\cdot)$ is a probability mass function.
\end{proof}
Finally, we are ready to state and prove the main result of this subsection. 
\begin{prop}\label{limitform}
For any $(t,x)\in \rr_{>0}\times \rr^d$,
\begin{align*}
f(t,x) &= \int_{\rr^d} K(t, x -y) g(y)dy \\
&\qquad +\gamma \int_0^{t} \int_{\rr^d} K(s, x-y) |\nabla f|^2(t-s, y)dy ds.
\end{align*}
\end{prop}
\begin{proof}
Fix $t$ and $x$. Take any $\delta\in (0,t/2)$. Define
\begin{align*}
&f^{(\ve)}_\delta (t,x) \\
&:= \int_{\rr^d} p^{(\ve)}(\ve^2t_\ve, \ve(x_\ve-y_\ve)) g^{(\ve)}(\ve y_\ve)dy \\
&\qquad + \int_\delta^t\int_{\rr^d} p^{(\ve)}(\ve^2 s_\ve, \ve(x_\ve-y_\ve)) h^{(\ve)}(\ve^2(t_\ve-s_\ve), \ve y_\ve)dy ds.
\end{align*}
Then by Lemma \ref{feform} and the facts that $\ve^2t_\ve \le t$, and $s_\ve \le t_\ve$ whenever $s\le t$, we get
\begin{align*}
&|f^{(\ve)}(t,x)-f^{(\ve)}_\delta(t,x) | \\
&\le \biggl|\int_0^\delta \int_{\rr^d} p^{(\ve)}(\ve^2 s_\ve, \ve(x_\ve-y_\ve)) h^{(\ve)}(\ve^2(t_\ve-s_\ve), \ve y_\ve)dy ds\biggr|\\
&\qquad + \biggl|\int_{\ve^2t_\ve}^t\int_{\rr^d} p^{(\ve)}(\ve^2 s_\ve, \ve(x_\ve-y_\ve)) h^{(\ve)}(\ve^2(t_\ve-s_\ve), \ve y_\ve)dy ds\biggr|.
\end{align*}
By Lemma \ref{lip2lmm}, $|h^{(\ve)}|$ is uniformly bounded by $C$. Therefore using the identity from  Lemma \ref{remlmm} and the above bound, we get
\begin{align}\label{ffdelta}
|f^{(\ve)}(t,x)-f^{(\ve)}_\delta(t,x) |  &\le C\delta + C\ve^2. 
\end{align}
Now, by Lemma \ref{plimlmm}, 
\begin{align*}
\lim_{\ve \to 0} p^{(\ve)}(\ve^2 s_\ve, \ve(x_\ve-y_\ve)) = K(s,x-y)
\end{align*}
for any $s\in \rr_{>0}$ and $x,y\in  \rr^d$. By the continuity of $g$, we have $g^{(\ve)}(\ve y_\ve) \to g(y)$ as $\ve \to 0$. Moreover, by the Lipschitz property of $g$, we have that $|g^{(\ve)}(\ve y_\ve)| \le C_1+C_2|y|$ where $C_1$ and $C_2$ do not depend on $\ve$. Therefore by Lemma \ref{domlmm} and the dominated convergence theorem, we get
\begin{align}\label{ffdelta1}
\lim_{\ve \to 0} \int_{\rr^d} p^{(\ve)}(\ve^2t_\ve, \ve(x_\ve-y_\ve)) g^{(\ve)}(\ve y_\ve)dy = \int_{\rr^d} K(t, x-y) g(y)dy.
\end{align}
Next, by Proposition \ref{hepprop}, we have that for any $t>s>0$ and $y\in \rr^d$,
\begin{align*}
\lim_{n\to\infty} h^{(\ve_n)}(\ve_n^2(t_{\ve_n}-s_{\ve_n}), \ve_n y_{\ve_n}) = \gamma|\nabla f|^2(t-s,y). 
\end{align*}
Moreover, by Lemma \ref{lip2lmm}, $h^{(\ve)}$ is uniformly bounded by a constant that does not depend on $\ve$. Therefore again by Lemma \ref{domlmm} and the dominated convergence theorem, 
\begin{align}
&\lim_{n\to\infty} \int_\delta^t\int_{\rr^d} p^{(\ve_n)}(\ve_n^2 s_{\ve_n}, \ve_n(x_{\ve_n}-y_{\ve_n})) h^{(\ve_n)}(\ve_n^2(t_{\ve_n}-s_{\ve_n}), \ve_n y_{\ve_n})dy ds\notag \\
&=  \gamma\int_\delta^t\int_{\rr^d} K(s, x-y) |\nabla f|^2(t-s,y)dy ds.\label{ffdelta2}
\end{align}
Combining \eqref{ffdelta}, \eqref{ffdelta1} and \eqref{ffdelta2}, we get
\begin{align}
&\biggl|f(t,x)- \int_{\rr^d} K(t, x -y) g(y)dy \notag \\
&\qquad- \gamma\int_\delta^{t} \int_{\rr^d} K(s, x-y) |\nabla f|^2(t-s, y)dy ds\biggr| \le C\delta. \label{duhamel0}
\end{align} 
Next, note that by Corollary \ref{diffprop}, $|\nabla f|^2$ is uniformly bounded by a constant. Also, for any $s$ and $x$,
\[
\int_{\rr^d} K(s, x-y)dy = 1. 
\]
This gives 
\[
\biggl|\int_0^\delta \int_{\rr^d} K(s, x-y) |\nabla f|^2(t-s, y)dy ds\biggr| \le C\delta. 
\]
Combining this with \eqref{duhamel0} gives
\begin{align*}
&\biggl|f(t,x)- \int_{\rr^d} K(t, x -y) g(y)dy \\
&\qquad- \gamma \int_0^{t} \int_{\rr^d} K(s, x-y) |\nabla f|^2(t-s, y)dy ds\biggr| \le C\delta. \end{align*} 
Since $\delta$ is arbitrary, this completes the proof.
\end{proof}

\subsection{Uniqueness of subsequential limit}
We will now show that the integral equation displayed in Proposition \ref{limitform} has at most one solution among all possible subsequential limits of $f^{(\ve)}$. 
\begin{prop}\label{uniqueprop}
There can be at most one function $f:\rr_{>0}\times \rr^d\to \rr$ that satisfies all of the following conditions: 
\begin{enumerate}
\item For each $t>0$, $f$ is differentiable in $x$.
\item The gradient of $f$ with respect to $x$, denoted by $\nabla f$, is uniformly bounded over $\rr_{>0}\times \rr^d$.
\item For every $(t,x)\in \rr_{>0}\times \rr^d$, 
\begin{align*}
f(t,x) &= \int_{\rr^d} K(t, x -y) g(y)dy \\
&\qquad +\gamma\int_0^{t} \int_{\rr^d} K(s, x-y) |\nabla f|^2(t-s, y)dy ds.
\end{align*}
\end{enumerate}
\end{prop}
\begin{proof}
Let $f_1$ and $f_2$ be two functions satisfying all of the given conditions. Let $u := f_1-f_2$ and $v := |\nabla f_1|^2 - |\nabla f_2|^2$, so that
\begin{align*}
u(t,x) &= \gamma \int_0^t \int_{\rr^d} K(s,x-y) v(t-s,y) dyds. 
\end{align*}
The first step is to show that when $t>0$, the right side of the above equation can be differentiated with respect to $x$ and the derivative can be moved inside the integral to give
\begin{align}\label{udiff}
\nabla u(t,x) &= \gamma \int_0^t \int_{\rr^d} \nabla K(s,x-y) v(t-s,y) dyds.
\end{align}
To prove this, define for each $\delta\in (0,t)$,
\begin{align*}
u_\delta(t,x) &:= \gamma \int_\delta^t \int_{\rr^d} K(s,x-y) v(t-s,y) dyds. 
\end{align*}
Since $v$ is a uniformly bounded function and $s$ is uniformly bounded away from zero in the above integral, it is not difficult to show by the dominated convergence theorem that $u_\delta$ is differentiable and 
\[
\nabla u_\delta(t,x) = \gamma \int_\delta^t \int_{\rr^d} \nabla K(s,x-y) v(t-s,y) dyds.
\]
Again, since $v$ is uniformly bounded, this shows that for any $\delta>\delta'>0$,
\begin{align*}
|\nabla u_\delta(t,x) - \nabla u_{\delta'}(t,x)| &\le C\int_{\delta'}^\delta \int_{\rr^d}|\nabla K(s,x-y)| dyds\\
&\le C\int_{\delta'}^\delta \int_{\rr^d}s^{-1}|x-y|K(s,x-y) dyds\le C\sqrt{\delta}. 
\end{align*}
Thus, $\nabla u_\delta(t,x)$ converges uniformly as $\delta\to0$. By a slight modification of the above display, it is easy to see that the limit equals the right side of \eqref{udiff}. Moreover, by the boundedness of $v$, $u_\delta(t,x)\to u(t,x)$ as $\delta\to 0$. From these observations and the standard condition for convergence of derivatives, we get \eqref{udiff}. 

Next, for each $s$, define 
\[
a(s) := \sup_{y\in \rr^d} |\nabla u(s,y)|, \ \  b(s) := \sup_{y\in \rr^d} |v(s,y)|.
\]
Note that these quantities are finite since $\nabla f_1$ and $\nabla f_2$ are uniformly bounded. Since $\nabla K(s,x-y) = (2\beta s)^{-1} K(s,x-y)  (x-y)$, the identity \eqref{udiff} gives us
\begin{align*}
a(t) &\le C\int_0^t \int_{\rr^d} s^{-1} |x-y|K(s,x-y) b(t-s) dy ds\\
&\le C\int_0^t s^{-1/2} b(t-s) ds = C \int_0^t (t-s)^{-1/2} b(s) ds.
\end{align*}
Now, since $\nabla f_1$ and $\nabla f_2$ are uniformly bounded,
\begin{align*}
|v(s,y)| &= ||\nabla f_1(s,y)|^2 - |\nabla f_2(s,y)|^2|\\
&\le C ||\nabla f_1(s,y)| - |\nabla f_2(s,y)||\\
&\le C|\nabla f_1(s,y)-  \nabla f_2(s,y)| = C|\nabla u(s,y)|. 
\end{align*}
Thus,
\begin{align*}
a(t) &\le C\int_0^t (t-s)^{-1/2} a(s) ds. 
\end{align*}
Fix some $T>0$. Then for any $0\le t\le T$, H\"older's inequality gives
\begin{align*}
a(t) &\le C \biggl(\int_0^t (t-s)^{-3/4} ds\biggr)^{2/3} \biggl(\int_0^t a(s)^3 ds\biggr)^{1/3}\\
&\le CT^{1/6}\biggl(\int_0^t a(s)^3 ds\biggr)^{1/3}. 
\end{align*}
Thus, for every $t\le T$, 
\[
a(t)^3 \le CT^{1/2}\int_0^t a(s)^3 ds. 
\]
Moreover, the function $a$ is uniformly bounded. So by Gr\"onwall's lemma, it now follows that $a(t)=0$ for all $t\in [0,T]$. Since $T$ is arbitrary, this completes the proof. 
\end{proof}

\subsection{The Cole--Hopf solution}
We will now complete the proof of Theorem \ref{kpzunivthm} under the condition that $\alpha \ne 0$. First, suppose that $\beta$ and $\gamma$ are both nonzero.  Let $b:= \gamma/\beta$. For $t>0$ and $x\in\rr^d$, define
\begin{align*}
u(t,x) := \int_{\rr^d} K(t,x-y) e^{bg(y)} dy = \int_{\rr^d} K(t,y) e^{bg(x-y)} dy.
\end{align*}
Then $u$ is strictly positive everywhere. Using the Lipschitzness of $g$ and the dominated convergence theorem, it is not difficult to verify $u$ is infinitely differentiable in $\rr_{>0}\times \rr^d$, continuous on the closure of this domain, and solves the heat equation
\[
\partial_t u = \beta \Delta u
\]
with initial condition $u(0,\cdot) = e^{b g(\cdot)}$. Define
\[
v(t,x) := \frac{1}{b}\log u(t,x).
\]
Then $v$ is also infinitely differentiable in $\rr_{>0}\times \rr^d$, continuous in $\rr_{\ge0} \times \rr^d$, and solves the equation
\begin{align}\label{ptv}
\partial_t v= \frac{\partial_t u}{b u} = \frac{\beta \Delta u}{b u} 
\end{align}
with initial condition $v(0,\cdot) = g(\cdot)$. Moreover, if $g$ is differentiable (in addition to being Lipschitz), then 
\begin{align*}
\nabla v(t,x) &= \frac{\int_{\rr^d} K(t,y) e^{bg(x-y)} \nabla g(x-y)dy}{\int_{\rr^d} K(t,y) e^{bg(x-y)} dy},
\end{align*}
which shows that $|\nabla v|$ is uniformly bounded by $L$ (the Lipschitz constant of $g$). If $g$ is only Lipschitz, and not differentiable, then it is not hard to show that $|\nabla v|$ is still uniformly bounded by $L$, by approximating $g$ with a sequence of differentiable Lipschitz maps. 

Now, from the definition of $v$, an easy calculation gives 
\begin{align}\label{deltav}
\Delta v &= \frac{\Delta u}{bu} - \frac{|\nabla u|^2}{bu^2}= \frac{\Delta u}{bu} - b|\nabla v|^2. 
\end{align}
Combining \eqref{ptv} and \eqref{ptv}, we get
\begin{align*}
\partial_t v = \beta \Delta v + \beta b |\nabla v|^2 = \beta \Delta v + \gamma|\nabla v|^2. 
\end{align*}
Next, for $(t,x)\in \rr_{>0}\times \rr^d$, define
\begin{align*}
w_1(t,x) &:= \int_{\rr^d} K(t,x-y) g(y)dy,\\
w_2(t,x) &:= \gamma \int_0^t \int_{\rr^d} K(t-s,x-y) |\nabla v|^2(s,y)dyds,\\
w(t,x) &:= v(t,x) -  w_1(t,x)-w_2(t,x). 
\end{align*}
Since $K$ is a heat kernel and $g$ is sufficiently well-behaved, it is easy to see that $w_1$ satisfies the heat equation
\begin{align*}
\partial_t w_1 =\beta \Delta w_1.
\end{align*}
Similarly, it is not hard to show using standard arguments, and the boundedness and continuity of $|\nabla v|$, that $w_2$ solves
\begin{align*}
\partial_t w_2(t,x) = \gamma |\nabla v|^2(t,x) + \beta \Delta w_2(t,x). 
\end{align*}
Thus,
\begin{align*}
\partial_t w &= \partial_t v - \partial_t w_1 - \partial_t w_2\\
&= (\beta \Delta v + \gamma|\nabla v|^2) - \beta \Delta w_1 - (\gamma|\nabla v|^2 + \beta \Delta w_2)\\
&=\beta \Delta w.  
\end{align*}
Now, $w_1$ extends to a continuous function on $\rr_{\ge0}\times \rr^d$ by defining $w_1(0,x)=g(x)$. By the boundedness of $|\nabla v|$, $w_2$ extends to a continuous function on $\rr_{\ge0}\times \rr^d$ by defining $w_2(0,x)=0$. Lastly, as observed before, $v$ also extends continuously, with $v(0,x)=x$.  Thus, $w$ extends continuously to the boundary with $w(0,x)=0$. From this, and the fact that $w$ solves the heat equation displayed above and is sufficiently well-behaved due to the Lipschitzness of $g$ and the boundedness of $|\nabla v|$, it is now a standard exercise to show that $w(t,x)=0$ for all $t$ and $x$. This is equivalent to saying that $v$ satisfies the integral equation of Proposition~\ref{uniqueprop}. We have already observed that $v$ satisfies the other two conditions of Proposition~\ref{uniqueprop}. Thus, by Proposition \ref{limitform}, every subsequential limit of $f^{(\ve)}$ must equal $v$. It is now easy to argue using Proposition \ref{subseqprop} that $f^{(\ve)}$ converges pointwise to~$v$. This completes the proof of Theorem \ref{kpzunivthm} when $\alpha$, $\beta$ and $\gamma$ are all nonzero.

Next, suppose that only $\alpha$ and $\beta$ are nonzero, but $\gamma=0$. Then define
\[
v(t,x) := \int_{\rr^d} K(t,x-y)g(y)dy. 
\]
Then immediately by Propositions \ref{limitform} and \ref{uniqueprop}, we see that any subsequential limit of $f^{(\ve)}$ must be equal to $v$, and thus conclude that $f^{(\ve)}$ converges pointwise to~$v$.

Finally, suppose that $\beta =0$. We claim that $\gamma_1$, $\gamma_2$ and $\gamma_3$ must all be zero in this case. Fix some $b\in B$. For any $\delta\in \rr$, let $u^\delta\in \rr^A$ be the vector defined as $u^\delta_a=0$ for $a\in A\setminus\{b\}$ and $u^\delta_b = \delta$. Since $\phi(0) =0$ and $\beta=0$, Taylor expansion gives
\[
\phi(u^\delta) = \frac{1}{2}\gamma_1 \delta^2 + o(\delta^2)
\]
as $\delta \to 0$. But this shows that $\phi(u^\delta)$ is not a monotone function of $\delta$ in a small enough neighborhood of $0$ if $\gamma_1\ne 0$. This violates the monotonicity of $\phi$. Thus, $\gamma_1$ must be zero. Next, define $v^\delta \in \rr^A$ as $v^\delta_b = v^\delta_{-b} =\delta$ and $v^\delta_a=0$ for $a\in A\setminus\{b,-b\}$. Since $\phi(0)=\beta=\gamma_1=0$, Taylor expansion gives
\[
\phi(v^\delta) = \gamma_2 \delta^2 + o(\delta^2)
\]
as $\delta\to 0$. Since $v^\delta$ is monotone increasing in $\delta$, the above expression violates the monotonicity of $\phi$ when $\delta$ is in a small enough neighborhood of zero if $\gamma_2\ne 0$. Thus, $\gamma_2=0$. Finally, take any $b,b'\in B$ such that $b\ne b'$ and $b\ne -b'$. Define $w^\delta \in \rr^A$ as $w^\delta_b = w^\delta_{b'} =\delta$ and $w^\delta_a=0$ for $a\in A\setminus\{b,b'\}$. Since $\phi(0)=\beta =\gamma_1=0$, Taylor expansion gives
\[
\phi(w^\delta) = \gamma_3 \delta^2 + o(\delta^2)
\]
as $\delta \to 0$. Proceeding as before, this yields $\gamma_3=0$.

Since $\beta=\gamma_1=\gamma_2=\gamma_3=0$, we get that for any $u\in \rr^A$ with $u_0=0$, $\phi(u) = o(|u|^2)$ as $|u|\to 0$. Since Lemma \ref{lipschitzlmm} is still valid, this shows that
\begin{align*}
|f_\ve(t,x) - f_\ve(t-1,x)| &= |\phi((f(t-1,x+a) - f(t-1,x))_{a\in A})| \\
&\le \ve^2F(\ve), 
\end{align*}
where $F(\ve)$ is a function of $\ve$ that tends to zero as $\ve \to 0$. Thus, for any $(t,x)\in \rr_{>0}\times \rr^d$,
\begin{align*}
|f^{(\ve)}(t,x)-g(x)| &= |f_\ve(t_\ve,x_\ve)-f_\ve(0,x_\ve)|\\
&\le \sum_{s=1}^{t_\ve} |f_\ve(s, x_\ve)-f_\ve(s-1,x_\ve)|\\
&\le t_\ve \ve^2F(\ve). 
\end{align*}
Since $t_\ve \ve^2\to t$ and $F(\ve)\to0$ as $\ve \to 0$, this shows that $f^{(\ve)}(t,x)\to g(x)$ as $\ve \to 0$. This wraps up the proof of Theorem \ref{kpzunivthm} when $\alpha \ne0$. 

\subsection{The case $\alpha = 0$}\label{alphazero}
We will now deal with the case $\alpha=0$. The main issue is that the random walk described in Subsection \ref{rwsec} is now periodic, which necessitates several changes in various steps of the proof. 
The results of Subsections \ref{prelimsec}, \ref{rwsec} and \ref{basicsec} continue to remain valid. A number of changes need to be made to the results of Subsection \ref{approxsec}. First, Lemma \ref{psi1} needs to be modified. Note that when $\alpha =0$, we have $\beta = 1/2d$, and hence
\[
\psi(\theta) = \ee(e^{\I \theta \cdot \xi}) = \frac{1}{d}\sum_{i=1}^d \cos \theta_i.
\]
Let $\upi$ be the element of $\rr^d$ whose coordinates are all equal to $\pi$. For $\theta \in \rr^d$, define
\begin{align*}
q(\theta) := \min\{|\theta|, |\theta-\upi|\}. 
\end{align*}
The following lemma is the new version of Lemma \ref{psi1}. 
\begin{lmm}\label{psi1b}
There is a positive constant $C$ depending only $d$, such that for all $\theta\in [-\pi/2,3\pi/2]^d$,
$|\psi(\theta)|\le e^{-Cq(\theta)^2}$.
\end{lmm}
\begin{proof}
Let $U$ denote the interval $[-\pi/2, 3\pi/2]$. Let $\delta$ be the smallest positive solution of $\cos x = 1-1/2d$. Let $I_1 := [-\delta,\delta]$, $I_2 = [\pi-\delta, \pi+\delta]$ and $I := I_1\cup I_2$. Note that  $|\cos x|\le 1- 1/2d$ for all $x\in U \setminus I$. Take any $\theta = (\theta_1,\ldots,\theta_d)\in U^d\setminus I^d$. Then $\theta_i\notin I$ for at least one $i$, and hence
\[
|\psi(\theta)| \le \frac{1}{d} \biggl(d-1+ 1-\frac{1}{2d}\biggr) = 1- \frac{1}{2d^2}. 
\]
Since $q$ is uniformly bounded in $U^d$, the above inequality shows that  there is a small enough positive constant $C_1$ such that $|\psi(\theta)|\le e^{-C_1q(\theta)^2}$ for all $\theta\in  U^d\setminus I^d$.

Next, take $\theta \in I^d$.  Let $J_1$ be the set of all $j$ such that $\theta_j\in I_1$, and let $J_2$ be the set of all $j$ such that $\theta_j\in I_2$. Suppose that $J_1$ and $J_2$ are both nonempty. Then, since $\cos x \le 1$ for all $x\in I_1$ and $\cos x \le -1+1/2d$ for all $x\in I_2$, we have 
\begin{align*}
\psi(\theta) &\le \frac{|J_1|}{d} + \biggl(-1+\frac{1}{2d}\biggr) \frac{|J_2|}{d}\\
&= \frac{|J_1|-|J_2|}{d} + \frac{|J_2|}{2d^2} \\
&\le \frac{d-2}{d} + \frac{d-1}{2d^2} =  1 - \frac{3d+1}{2d^2},
\end{align*}
and similarly, since $\cos x \ge 1-1/2d$ for all $x\in I_1$ and $\cos x \ge -1$ for all $x\in I_2$,
\begin{align*}
\psi(\theta) &\ge \biggl(1-\frac{1}{2d}\biggr)\frac{|J_1|}{d} -\frac{|J_2|}{d}\\
&= \frac{|J_1|-|J_2|}{d} -\frac{|J_1|}{2d^2}\\
&\ge \frac{2-d}{d} - \frac{d-1}{2d^2} = -1 + \frac{3d+1}{2d^2}. 
\end{align*}
Thus, there is a positive constant $C_2$ such that if the sets $J_1$ and $J_2$ are both nonempty, then $|\psi(\theta)|\le e^{-C_2q(\theta)^2}$. 

Next, suppose that $J_2$ is empty. Then $\theta_i\in I_1$ for each $i$. There is a positive constant $C_3$ such that $|\cos x|\le 1 - C_3x^2$ for all $x\in I_1$. Therefore, in this case,
\begin{align*}
|\psi(\theta)| &\le \frac{1}{d}\sum_{i=1}^d |\cos \theta_i|\\
&\le \frac{1}{d}\sum_{i=1}^d (1-C_3 \theta_i^2)\\
&= 1- C_3 |\theta|^2 \le e^{-C_3|\theta|^2}\le e^{-C_3q(\theta)^2}.  
\end{align*}
Finally, suppose that $J_1$ is empty, so that $\theta_i \in I_2$ for each $i$. Observe that $|\cos x| = |\cos (x-\pi)|\le 1-C_3 (x-\pi)^2$ for all $x\in I_2$. Thus, as above, we have
\[
|\psi(\theta)|\le e^{-C_3|\theta - \upi|^2} \le e^{-C_3q(\theta)^2}. 
\]
So, if $C := \min\{C_1,C_2,C_3\}$, then $|\psi(\theta)|\le e^{-Cq(\theta)^2}$ for all $\theta \in U^d$.
\end{proof}
Lemma \ref{psi2} also needs to be modified. The following is the new version of Lemma \ref{psi2}.
\begin{lmm}\label{psi2b}
For any $\theta \in \rr^d$ and $t\in \zz_{\ge 1}$, we have
\begin{align*}
&|\psi(\theta)^t - e^{-\beta t |\theta|^2}|\le Ct|\theta|^4,\\
&|\psi(\upi + \theta)^t - (-1)^te^{-\beta t |\theta|^2}|\le Ct |\theta|^4. 
\end{align*}
\end{lmm}
\begin{proof}
The proof of the first inequality is just the same as the proof of Lemma \ref{psi2}. For the second, notice that 
\[
\psi(\upi + \theta) = \frac{1}{d}\sum_{i=1}^d \cos(\pi+\theta_i) = -\frac{1}{d}\sum_{i=1}^d \cos\theta_i = -\psi(\theta).
\]
So, by the first inequality applied with $t=1$, we have
\[
|\psi(\upi+\theta) + e^{-\beta |\theta|^2 }| = |-\psi(\theta) + e^{-\beta|\theta|^2}| \le C|\theta|^4.
\]
But recall that if $a$ and $b$ are complex numbers in the unit disk, and $t$ is a positive integer, then $|a^t - b^t|\le t |a-b|$. Applying this in the above inequality, with $a = \psi(\upi+\theta)$ and $b = -e^{-\beta |\theta|^2 }$, we get the desired result.
\end{proof}

Next, we introduce a different form of the Fourier inversion formula for transition probabilities: 
\begin{align}
p(t,x) &= (2\pi)^{-d}\int_{[-\pi/2, 3\pi/2]^d} e^{-\I \theta \cdot x} \psi(\theta)^t d\theta\notag\\
&= (2\pi)^{-d} t^{-d/2}\int_{[-\pi\sqrt{t}/2, 3\pi\sqrt{t}/2]^d} e^{-\I t^{-1/2}\eta \cdot x}(\psi(t^{-1/2}\eta ))^t d\eta,\label{finv2}
\end{align}
The formula for $p_1$ also needs to be changed. We define
\begin{align*}
p_1(t,x) &:= (2\pi)^{-d} t^{-d/2}\int_{|\eta|\le \log t} e^{-\I t^{-1/2}\eta \cdot x}(\psi(t^{-1/2}\eta ))^t d\eta\\
&\qquad + (2\pi)^{-d} t^{-d/2}\int_{|\eta-\sqrt{t}\upi|\le \log t} e^{-\I t^{-1/2}\eta \cdot x}(\psi(t^{-1/2}\eta ))^t d\eta.
\end{align*}
The following is the new version of Lemma \ref{pp1lmm}. 
\begin{lmm}\label{pp1lmmb}
For any $t\in \zz_{\ge 1}$ and $x\in \zz^d$, $|p(t,x)-p_1(t,x)|\le C_1 e^{-C_2 (\log t)^2}$. 
\end{lmm}
\begin{proof}
Without loss, let us assume that $t$ is so large that the regions $|\eta|\le \log t$ and $|\eta - \sqrt{t}\upi|\le \log t$ do not intersect. Let us take $\eta$ outside the union of these two regions. Then note that 
\begin{align*}
q(t^{-1/2}\eta) &=\min\{|t^{-1/2}\eta|, |t^{-1/2}\eta - \upi|\}\ge t^{-1/2}|\log t|.
\end{align*}
Therefore, by Lemma \ref{psi1b}, 
\begin{align*}
|\psi(t^{-1/2} \eta)| &\le e^{-Cq(t^{-1/2}\eta)^2} \le e^{-Ct^{-1}(\log t)^2}. 
\end{align*}
Now proceeding as in the proof of Lemma \ref{pp1lmm}, we get the required bound.
\end{proof}
For $x= (x_1,\ldots,x_d)\in \rr^d$, define $s(x) := \sum_{i=1}^d x_i$. 
We will say that $x$ is even if $s(x)$ is even, and $x$ is odd if $s(x)$ is odd. We define $p_2(t,x)$ and $p_3(t,x)$ to be twice the values defined in Subsection \ref{approxsec} if $t$ and $x$ have the same parity, and equal to zero if not. With these new definitions, Lemma \ref{p2p3lmm} clearly remains valid. 

Next, we modify the results of Subsection \ref{transsec}. The following is the new version of Lemma~\ref{p1p2lmm}. 
\begin{lmm}\label{p1p2lmmb}
For any $t\ge \zz_{\ge 1}$ and $x,y,z\in \zz^d$ such that $x$ and $t$ have the same parity, and $y$ and $z$ are even, we have
\begin{align*}
|p_1(t,x)-p_2(t,x)|&\le C_1 t^{-(d+2)/2}(\log t)^{C_2},\\
|\delta_yp_1(t,x)-\delta_yp_2(t,x)|&\le C_1|y| t^{-(d+3)/2}(\log t)^{C_2},\\
|\delta_y\delta_zp_1(t,x)-\delta_y\delta_zp_2(t,x)|&\le C_1|y||z| t^{-(d+4)/2}(\log t)^{C_2}.
\end{align*}
\end{lmm}
\begin{proof}
Note that
\begin{align*}
p_1(t,x) &= (2\pi)^{-d} t^{-d/2}\int_{|\eta|\le \log t} e^{-\I t^{-1/2}\eta \cdot x}(\psi(t^{-1/2}\eta ))^t d\eta\\
&\qquad + (2\pi)^{-d} t^{-d/2}\int_{|\eta|\le \log t} e^{-\I(\upi+ t^{-1/2}\eta )\cdot x}(\psi(\upi + t^{-1/2}\eta ))^t d\eta\\
&= (2\pi)^{-d} t^{-d/2}\int_{|\eta|\le \log t} e^{-\I t^{-1/2}\eta \cdot x}(\psi(t^{-1/2}\eta ))^t d\eta\\
&\qquad + (2\pi)^{-d} t^{-d/2}(-1)^{s(x)}\int_{|\eta|\le \log t} e^{-\I t^{-1/2}\eta \cdot x}(\psi(\upi + t^{-1/2}\eta ))^t d\eta
\end{align*}
By Lemma \ref{psi2b}, 
\begin{align*}
&|(\psi(t^{-1/2}\eta ))^t - e^{-\beta |\eta|^2}| \le C t^{-1} |\eta|^4,\\
&|(\psi(\upi + t^{-1/2}\eta ))^t - (-1)^t e^{-\beta |\eta|^2}| \le C t^{-1} |\eta|^4.
\end{align*}
Lastly, note that
\begin{align*}
p_2(t,x) &=(1+(-1)^{s(x)+t})(2\pi)^{-d} t^{-d/2}\int_{|\eta|\le \log t} e^{-\I t^{-1/2}\eta \cdot x - \beta |\eta|^2} d\eta
\end{align*}
Combining the last three displays, we get
\begin{align*}
|p_1(t,x)-p_2(t,x)| &\le Ct^{-d/2}\int_{|\eta|\le \log t} |(\psi(t^{-1/2}\eta ))^t - e^{-\beta |\eta|^2}|  d\eta\\
&\qquad + Ct^{-d/2}\int_{|\eta|\le \log t} |(\psi(t^{-1/2}\eta ))^t - (-1)^te^{-\beta |\eta|^2}|  d\eta\\
&\le Ct^{-d/2-1} \int_{|\eta|\le \log t} |\eta|^4 d\eta\\
&\le C t^{-d/2-1}(\log t)^{d+4}.
\end{align*}
This proves the first inequality in the statement of the lemma. Next, note that since $y$ is even, we have $e^{-\I \upi\cdot y} =1$. This gives 
\begin{align*}
&\delta_yp_1(t,x) \\
&= (2\pi)^{-d}t^{-d/2}\int_{|\eta|\le \log t}  (e^{-\I t^{-1/2}\eta \cdot y} - 1)e^{-\I t^{-1/2}\eta \cdot x}(\psi(t^{-1/2}\eta ))^t d\eta\\
&\qquad + (2\pi)^{-d}t^{-d/2}(-1)^{s(x)}\int_{|\eta|\le \log t}  (e^{-\I t^{-1/2}\eta \cdot y} - 1)e^{-\I t^{-1/2}\eta \cdot x}\\
&\qquad \qquad \qquad \qquad \qquad \qquad\qquad \qquad \cdot(\psi(\upi + t^{-1/2}\eta ))^t d\eta.
\end{align*}
Moreover, the evenness of $y$ ensures that $t$ and $x$ have the same parity if and only if $t$ and $x+y$ have the same parity. Thus,
\begin{align*}
&\delta_yp_2(t,x) \\
&= (1+(-1)^{s(x)+t})(2\pi)^{-d}t^{-d/2}\int_{|\eta|\le \log t} (e^{-\I t^{-1/2}\eta \cdot y} - 1) e^{-\I t^{-1/2}\eta \cdot x - \beta |\eta|^2} d\eta.
\end{align*}
It is now easy to complete the rest of the proof following the same steps as in the proof of Lemma \ref{p1p2lmmb} and implementing the same modifications as above. 
\end{proof}
As a result of Lemma \ref{p1p2lmmb}, the remaining results of Subsections \ref{transsec} and \ref{surfacederivsec}  remain valid after  the following modifications:
\begin{itemize}
\item In Lemma \ref{p3lmm}, we have to insert the additional condition that $x$ and $t$ have the same parity, and $y$ and $z$ are even. 
\item Lemma \ref{plmm} can remain as it is, with the extra observation that $p(t,x)=0$ when $t$ and $x$ do not have the same parity.
\item Proposition \ref{delprop} remains valid if $y$ is even. We do not need to add the condition and that $x$ and $t$ have the same parity, since $\delta_y p(t,x)=0$ if that is not the case (and $y$ is even). 
\item Similarly, Proposition \ref{pprop} remains valid if $y$ and $z$ are even.
\item In Lemma \ref{p3lmm2}, we need $r$ to be even, and in Lemma \ref{p3lmm3}, we need both $r$ and $y$ to be even.
\item Similarly, in Proposition \ref{plmm2}, we need $r$ to be even, and in Proposition~\ref{plmm3}, we need both $r$ and $y$ to be even.
\item Because of the above changes, we need $y$ and $z$ to be even in Proposition~\ref{secondderiv}. Similarly, we need $r$ to be even in Proposition \ref{timeprop1}, and  both $r$ and $y$ to be even in Proposition \ref{timeprop2}. 
\item Proposition \ref{bdrycont} remains valid without any changes.
\end{itemize}
From Subsection \ref{subseqsec} onwards, we have to be cautious about the definition of $f^{(\ve)}$. We will work with two different definitions, and show that they both converge to the same limit. From that, we will deduce that the original $f^{(\ve)}$ also converges to that limit. 

First, let us define for a real number $a$, 
\[
[a]^0 := 2[a/2].
\]
This map takes the interval $[2k, 2k+2)$ to the number $2k$, for any $k\in \zz$. Next, let
\[
[a]^1 := [a]^0 + 1. 
\]
This map takes $[2k,2k+2)$ to $2k+1$. For $x=(x_1,\ldots,x_d)\in \rr^d$, define
\[
[x]^0 := 
\begin{cases}
([x_1]^0, [x_2],[x_3],\ldots,[x_d])  &\text{ if } [x_2]+\cdots+[x_d] \text{ is even,}\\
([x_1]^1, [x_2],[x_3],\ldots,[x_d]) &\text{ otherwise.}
\end{cases}
\]
Similarly, define
\[
[x]^1 := 
\begin{cases}
([x_1]^1, [x_2],[x_3],\ldots,[x_d])  &\text{ if } [x_2]+\cdots+[x_d] \text{ is even,}\\
([x_1]^0, [x_2],[x_3],\ldots,[x_d]) &\text{ otherwise.}
\end{cases}
\]
Note that the definitions are meant to ensure that $[x]^0$ is always even and $[x]^1$ is always odd.  Note also that for any $x$, $[x]$ is either equal to $[x]^0$ or equal to $[x]^1$. This observation will be useful in the following construction. Define two new versions of $f^{(\ve)}$, called $f^{(\ve,0)}$ and $f^{(\ve, 1)}$, as follows. Let
\begin{align*}
f^{(\ve,0)}(t,x) &:= 
\begin{cases}
f_\ve([\ve^{-2}t], [\ve^{-1}x]^0) &\text{ if $[\ve^{-2}t]$ is even,}\\
f_\ve([\ve^{-2}t], [\ve^{-1}x]^1) &\text{ if $[\ve^{-2}t]$ is odd,}
\end{cases}
\end{align*}
and let
\begin{align*}
f^{(\ve,1)}(t,x) &:= 
\begin{cases}
f_\ve([\ve^{-2}t], [\ve^{-1}x]^1) &\text{ if $[\ve^{-2}t]$ is even,}\\
f_\ve([\ve^{-2}t], [\ve^{-1}x]^0) &\text{ if $[\ve^{-2}t]$ is odd.}
\end{cases}
\end{align*}
We will show that both $f^{(\ve,0)}$ and $f^{(\ve,1)}$ converge pointwise to the same $f$ as $\ve \to 0$. Now note that by the observation from the previous paragraph, we have that for any $t$, $x$, and $\ve$, $f^{(\ve)}(t,x)$ has to be equal to either $f^{(\ve,0)}(t,x)$ or $f^{(\ve,1)}(t,x)$. Thus, it will follow that $f^{(\ve)}(t,x)$ must also be converging to $f(t,x)$ as $\ve \to 0$. 

Let us now prove the pointwise convergence of $f^{(\ve,0)}$ to $f$. The proof for $f^{(\ve,1)}$ is similar. 

First, note that Lemma \ref{contlmm1} remains valid with $f^{(\ve,0)}$ instead of $f^{(\ve)}$. Same is true for Lemma~\ref{contlmm2}, but special care has to be taken because the original Lemma~\ref{contlmm2} uses Proposition~\ref{timeprop1}, and in the modified version of Proposition~\ref{timeprop1}, we need $r$ to be even. The proof of Lemma~\ref{contlmm2} goes through as before when $r_\ve$ is even. However, $r_\ve$ is not guaranteed to be even. If for some $\ve$, $r_\ve$ turns out to be odd, we then proceed as follows. Suppose that $s_\ve$ is even and $t_\ve$ is odd. Let $x_\ve := [\ve^{-1}x]^0$ and $x_\ve' := [\ve^{-1}x]^1$. Then 
\begin{align*}
f^{(\ve)}(s,x) - f^{(\ve)}(t,x) &= f_\ve(s_\ve, x_\ve) - f_\ve(t_\ve, x_\ve'). 
\end{align*}
By Lemma \ref{lip2lmm}, 
\begin{align*}
\biggl|f_\ve(t_\ve, x_\ve') - \frac{1}{2d}\sum_{b\in B} f_\ve(t_\ve-1, x_\ve' + b)\biggr| &= |h_\ve(t_\ve, x_\ve')|\le C\ve^2. 
\end{align*}
Since $t_\ve-1-s_\ve$ is even, we can now use Proposition \ref{timeprop1} and Lemma \ref{lipschitzlmm} to bound the differences $|f_\ve(t_\ve-1, x_\ve' + b) - f_\ve(s_\ve, x_\ve)|$. Together with the above inequality, this suffices to complete the proof of Lemma \ref{contlmm2} for $f^{(\ve,0)}$.  Proposition \ref{subseqprop} for $f^{(\ve,0)}$ follows from this.

In Subsection \ref{diffsec}, we need to change the definitions of $D_i f_\ve$ and $D_i f^{(\ve)}$.  We define
\[
D_i f_\ve(t,x) := \frac{f_\ve(t,x+2e_i)-f_\ve(t, x)}{2\ve}, 
\]
and for $t\in \rr_{\ge0}$ and $x\in \rr^d$, let
\[
D_i f^{(\ve,0)}(t,x) := 
\begin{cases}
D_i f_\ve([\ve^{-2} t], [\ve^{-1} x]^0) &\text{ if $[\ve^{-2}t]$ is even,}\\
D_i f_\ve([\ve^{-2} t], [\ve^{-1} x]^1) &\text{ otherwise.}
\end{cases}
\]
The results of Subsection \ref{diffsec} remain valid after the following modifications (adopting the new definition of $D_i f_\ve$, and replacing $D_i f^{(\ve)}$ by $D_i f^{(\ve,0)}$ everywhere):
\begin{itemize}
\item Lemmas \ref{lip3lmm} and \ref{jlemma} remain as is.
\item Lemmas \ref{discderiv} and \ref{discderiv2} need the additional condition that $k$ is even. 
\item Lemma \ref{discderiv3} remains valid as is, because the number $k_\ve$ in its proof is now guaranteed to be even, so that the preceding two lemmas can be applied. (In the proof, we have to change $x_\ve$ to denote $[\ve^{-1}x]^0$ if $t_\ve$ is even, and $[\ve^{-1}x]^1$ if $t_\ve$ is odd. Similarly, $k_\ve$ has to be defined as the integer such that $[\ve^{-1} (x + a e_i)]^0 = x_\ve + k_\ve e_i$ if $t_\ve$ is even and the integer such that $[\ve^{-1} (x + a e_i)]^1 = x_\ve + k_\ve e_i$ if $t_\ve$ is odd.)
\item Consequently, Lemmas \ref{discderiv4}, \ref{difflmm},  Proposition~\ref{diffcont} and Corollary~\ref{diffprop} remain valid. In the proof of Proposition~\ref{diffcont}, we again run into the slight difficulty --- as in the proof of the new version of Lemma \ref{contlmm2} --- that $r_\ve$ need not be even, so Proposition \ref{timeprop2} may not be directly applicable. The difficulty is overcome easily by the same trick as we used for Lemma \ref{contlmm2}.
\end{itemize}
We need some further caution in Subsection \ref{gradsec}. First, for $a\in A$ and $(t,x)\in \zz_{\ge 0}\times \zz^d$, define
\[
u_{a,\ve}(t,x) := f_\ve(t,x+a) - \frac{1}{2d}\sum_{b\in B} f_\ve(t,x+b). 
\]
Let $u_\ve(t,x)$ denote the vector $(u_{a,\ve}(t,x))_{a\in A}$. Then redefine $H_\ve$ as 
\begin{align*}
H_\ve(t,x) &:= \frac{\gamma_1}{2}\sum_{b\in B} (u_{b,\ve}(t-1,x))^2 + \frac{\gamma_2}{2}\sum_{b\in B} u_{b,\ve}(t-1,x)u_{-b,\ve}(t-1,x) \\
&\qquad + \frac{\gamma_3}{2}\sum_{\substack{b,b'\in B \\ b\ne b', b\ne -b'}} u_{b,\ve}(t-1,x)u_{b',\ve}(t-1,x). 
\end{align*}
We will show that Lemma \ref{hehe} remains valid with this new $H_\ve$ (but the same old $h_\ve$). For that, we need some preparation. 
\begin{lmm}\label{secondzero}
If $\alpha =0$, then we must have that $\partial_0\partial_{a}\phi(0)=0$ for all $a\in A$. 
\end{lmm}
\begin{proof}
Take any $a\in A$. First, suppose that $a = 0$. Take any $\delta\in \rr$. Let $u^\delta$ be the vector in $\rr^A$ with $u_0^\delta = \delta$ and $u_b^\delta =0$ for all $b\ne 0$. Then by Taylor expansion and the facts that $\phi(0)=0$ and $\partial_0\phi(0)=\alpha = 0$, we have
\begin{align*}
\phi(u_\delta) := \frac{\delta^2}{2}\partial_0^2 \phi(0) + o(\delta^2)
\end{align*}
as $\delta \to 0$. If $\partial_0^2\phi(0)\ne0$, this is not a monotone function of $\delta$ in a small enough neighborhood of zero. This contradicts the monotonicity of $\phi$. Thus, $\partial_0^2\phi(0)=0$.

Next, suppose that $a\ne 0$. Given $\delta,\delta'\in \rr$, let $u^{\delta,\delta'}\in \rr^A$ be the vector with $u^{\delta,\delta'}_0=\delta$, $u^{\delta, \delta'}_a=\delta'$, and $u^{\delta, \delta'}_b = 0$ if $b\ne0$ and $b\ne a$. Then, since we already showed that $\partial_0^2 \phi(0)=0$, and we are given that $\phi(0)=0$ and $\partial_0\phi(0)=0$, Taylor expansion gives
\begin{align*}
\phi(u^{\delta, \delta'}) &= \delta' \partial_a\phi(0) + \frac{{\delta'}^2}{2}\partial_a^2 \phi(0) + \delta\delta'\partial_0\partial_a \phi(0) + o(\max\{\delta^2, {\delta'}^2\})
\end{align*}
as $\delta, \delta'\to 0$. Thus,
\begin{align*}
\phi(u^{2\delta, \delta'}) - \phi(u^{\delta, \delta'})  &= \delta\delta'\partial_0\partial_a \phi(0) + o(\max\{\delta^2, {\delta'}^2\})
\end{align*}
as $\delta, \delta'\to 0$.
Now choose $\delta' = -\delta \sign(\partial_0\partial_a \phi(0))$. With this choice of $\delta'$, we get
\begin{align*}
\phi(u^{2\delta, \delta'}) - \phi(u^{\delta, \delta'})  &= - \delta^2|\partial_0\partial_a \phi(0)| + o(\delta^2)
\end{align*}
as $\delta \to 0$. Take $\delta$ decreasing to zero. Then the left side is always nonnegative, by the monotone increasing nature of $\phi$; and the right side is negative for small enough $\delta$, unless $\partial_0\partial_a \phi(0)=0$. Thus, $\partial_0\partial_a \phi(0)$ must be zero.
\end{proof}
The following result is our modified version of Lemma \ref{hehe}, that works for the modified $H_\ve$ defined above. 
\begin{lmm}\label{hehe2}
For any $\ve>0$, $t\in \zz_{\ge 1}$ and $x\in \zz^d$,
\begin{align*}
|h_\ve(t,x) - H_\ve(t,x)| &\le \ve^2F(\ve),
\end{align*}
where $F$ is a function determined by $\phi$ (and not depending on $t$ or $x$), such that $F(\ve)\to 0$ as $\ve \to 0$. 
\end{lmm}
\begin{proof}
Note that
\begin{align*}
h_\ve(t,x) &= f_\ve(t,x) - \frac{1}{2d}\sum_{b\in B} f_\ve(t-1,x+b)\\
&= \phi(u_\ve(t-1,x)). 
\end{align*}
By Lemma \ref{lipschitzlmm}, $|u_{a,\ve}(t,x)|\le C\ve$ for all $a$, $t$ and $x$. We know that $\phi(0)=0$, $\partial_0\phi(0)=0$, $\partial_b \phi(0)=1/2d$ for all $b\in B$, and by Lemma \ref{secondzero}, $\partial_0\partial_a\phi(0)=0$ for all $a\in A$. Thus, by Taylor expansion, 
\begin{align*}
&\biggl|h_\ve(t,x) - \frac{1}{2d}\sum_{b\in B} u_{b,\ve}(t-1,x) \\
&\qquad - \frac{1}{2}\sum_{b,b'\in B} \partial_b\partial_{b'} \phi(0) u_{b,\ve}(t-1,x)u_{b',\ve}(t-1,x)\biggr|\le C\ve^2 F(\ve),
\end{align*}
where $F$ is determined solely by $\phi$, and $F(\ve)\to 0$ as $\ve \to 0$. Now,
\begin{align*}
\frac{1}{2d}\sum_{b\in B} u_{b,\ve}(t-1,x) &= \frac{1}{2d}\sum_{b\in B}\biggl( f_\ve(t-1,x+b) - \frac{1}{2d}\sum_{b'\in B} f_\ve(t-1,x+b')\biggr)\\
&= 0. 
\end{align*}
On the other hand,
\begin{align*}
\frac{1}{2}\sum_{b,b'\in B} \partial_b\partial_{b'} \phi(0) u_{b,\ve}(t-1,x)u_{b',\ve}(t-1,x) &= H_\ve(t,x). 
\end{align*}
Combining the last three displays completes the proof. 
\end{proof}
The next goal is to prove a modification of Proposition \ref{hepprop}. Again, some preparation is needed. We need the following enhancement of Lemma \ref{discderiv2}.
\begin{lmm}\label{discderiv2b}
For any $t\in \zz_{\ge 1}$, $x,y\in \zz^d$ with $y$ even, and any $k\in \zz\setminus\{0\}$,  
\begin{align*}
|f_\ve(t, x + ky) - f_\ve(t,x) - k\delta_yf_\ve(t,x)|  \le C(y)\ve |k| J(\ve, k,t,x) + C(y) \ve^2k^2,
\end{align*}
where $J$ is as in Lemma \ref{discderiv}, and $C(y)$ is a constant that depends only on $\phi$, $L$, $d$ and $y$. 
\end{lmm}
\begin{proof}
Suppose that $k>0$. Then note that 
\begin{align*}
&f_\ve(t, x + ky) - f_\ve(t,x) - k \delta_yf_\ve(x) \\
&= \sum_{j=1}^{k} (\delta_y f_\ve(x+jy) - \delta_y f_\ve(x))\\
&= \sum_{j=1}^{k} \delta_y\delta_{jy} f_\ve(x). 
\end{align*}
Observe  that $y$ and $jy$ are even. It is now easy to complete the proof using a suitable modification of Lemma \ref{discderiv}, and Lemma \ref{jlemma}. The proof for $k<0$ is similar. 
\end{proof}
\begin{lmm}\label{weirdderiv}
Fix some even $y\in \zz^d$. Take any sequence $(t_n,x_n)$ in $\zz_{>0}\times\zz^d$ such that $t_n$ has the same parity as $x_n$ for each $n$, and $(\ve_n^2 t_n, \ve_n x_n) \to (t,x)\in \rr_{>0}\times\rr^d$. Then
\[
\lim_{n\to\infty} \ve_n^{-1}\delta_yf_{\ve_n}(t_n,x_n) = y\cdot \nabla f(t,x).
\]
\end{lmm}
\begin{proof}
Take any $a>0$. Let $k_n := [a/\ve_n]$. By Lemma \ref{discderiv2b}, 
\begin{align*}
&\biggl|\frac{f_{\ve_n}(t_n, x_n + k_ny) - f_{\ve_n}(t_n,x_n)}{k_n\ve_n} - \ve_n^{-1}\delta_yf_{\ve_n}(t_n,x_n)\biggr|  \\
&\le C(y)J(\ve_n, k_n,t_n,x_n) + C(y) \ve_n k_n.
\end{align*}
Now notice that by the uniform convergence assertion of the version of Proposition~\ref{subseqprop} for $f^{(\ve,0)}$, we have
\begin{align*}
\lim_{n\to\infty} f_{\ve_n}(t_n,x_n) &= \lim_{n\to \infty} f^{(\ve_n,0)}(\ve_n^2 t_n, \ve_n x_n) = f(t,x),
\end{align*}
and similarly,
\begin{align*}
\lim_{n\to\infty} f_{\ve_n}(t_n, x_n + k_ny)   &= f(t, x+ay). 
\end{align*}
Also, $k_n\ve_n \to a$. Finally, note that
\begin{align*}
&\lim_{n\to\infty}J(\ve_n, k_n,t_n,x_n) \\
&= C_1 a\log (2+t/a^2)  + C_1(1+|x| +|a|+t^{1/2})t^{-1/2}\min\{1, t^{-1/2}|a|\}.
\end{align*}
Note that this is the quantity $Q(a,t,x)$ defined in Lemma \ref{discderiv3}. Thus, we get
\begin{align*}
\limsup_{n\to\infty} \biggl|\frac{f(t, x+ay)-f(t,x)}{a} - \ve_n^{-1}\delta_yf_{\ve_n}(t_n,x_n)\biggr|\le C(y)Q(a,t,x). 
\end{align*}
Since $Q(a,t,x)\to 0$ and $a^{-1}(f(t, x+ay)-f(t,x)) \to y\cdot \nabla f(t,x)$ as $a\to 0$, this completes the proof of the lemma.
\end{proof}
\begin{lmm}\label{uelim}
Take any $b\in B$, and any sequence $(t_n,x_n)$ in $\zz_{>0}\times\zz^d$ such that $t_n$ and $x_n$ have opposite parities for each $n$, and $(\ve_n^2 t_n, \ve_n x_n) \to (t,x)\in \rr_{>0}\times\rr^d$. Then
\begin{align*}
\lim_{n\to \infty} \ve_n^{-1}u_{b,\ve_n}(t_n,x_n) &= b\cdot \nabla f(t,x)
\end{align*}
\end{lmm}
\begin{proof}
Note that for any $\ve$, $t$, $x$ and $b$, 
\begin{align*}
u_{b,\ve}(t,x) &= f_\ve(t,x+b) - \frac{1}{2d}\sum_{b'\in B} f_\ve(t,x+b')\\
&= \frac{1}{2d}\sum_{b'\in B} (f_\ve(t,x+b)-f_\ve(t,x+b'))\\
&= \frac{1}{2d}\sum_{b'\in B} \delta_{b-b'} f_\ve(t, x+b'). 
\end{align*}
Thus, by Lemma \ref{weirdderiv}, we have 
\begin{align*}
\lim_{n\to \infty} \ve_n^{-1}u_{b,\ve_n}(t,x)  &= \frac{1}{2d}\sum_{b'\in B} \lim_{n\to\infty} \ve_n^{-1}\delta_{b-b'} f_{\ve_n}(t_n, x_n+b')\\
&= \frac{1}{2d}\sum_{b'\in B} (b-b')\cdot \nabla f(t,x)\\
&= b\cdot \nabla f(t,x).
\end{align*}
This completes the proof of the lemma.
\end{proof}
\begin{lmm}\label{heppreplmm}
Take any sequence $(t_n,x_n)$ in $\zz_{>0}\times\zz^d$ such that $t_n$ has the same parity as $x_n$ for each $n$, and $(\ve_n^2 t_n, \ve_n x_n) \to (t,x)\in \rr_{>0}\times\rr^d$. Then 
\begin{align*}
\lim_{n\to \infty} \ve_n^{-2}H_{\ve_n}(t_n,x_n) &= \gamma|\nabla f(t,x)|^2. 
\end{align*}
\end{lmm}
\begin{proof}
Note that 
\begin{align*}
&\ve_n^{-2}H_{\ve_n}(t_n,x_n) \\
&:= \frac{\gamma_1}{2}\sum_{b\in B} (\ve_n^{-1}u_{b,\ve_n}(t_n-1,x_n))^2 \\
&\qquad + \frac{\gamma_2}{2}\sum_{b\in B} (\ve_n^{-1}u_{b,\ve_n}(t_n-1,x_n))(\ve_n^{-1}u_{-b,\ve_n}(t_n-1,x_n)) \\
&\qquad + \frac{\gamma_3}{2}\sum_{\substack{b,b'\in B \\ b\ne b', b\ne -b'}} (\ve_n^{-1}u_{b,\ve_n}(t_n-1,x_n))(\ve_n^{-1}u_{b',\ve_n}(t_n -1,x_n)). 
\end{align*}
Applying Lemma \ref{uelim} to each term, we get
\begin{align*}
\lim_{n\to\infty} \ve_n^{-2}H_{\ve_n}(t_n,x_n) &= \frac{\gamma}{2}\sum_{b\in B} (b\cdot \nabla f(t,x))^2\\
&\qquad + \frac{\gamma_3}{2}\sum_{\substack{b,b'\in B \\ b\ne b', b\ne -b'}}(b\cdot \nabla f(t,x)) (b'\cdot \nabla f(t,x)). 
\end{align*}
Now note that for any $b\in B$, the sum of all $b'\in B\setminus\{b,-b\}$ is zero. Also, note that 
\[
\sum_{b\in B} (b\cdot \nabla f(t,x))^2 = 2|\nabla f(t,x)|^2.
\]
This completes the proof of the lemma.
\end{proof}
Next, we turn to modifying the results of Subsection \ref{duhamelsec}. For each even $x\in \zz^d$, there is a region of volume $2$ in $\rr^d$ that maps to $x$ under the map $y\mapsto [y]^0$. These regions form a partition of $\rr^d$. From this, it follows that for any $w:\zz^d \to \rr$, 
\begin{align*}
\sum_{\substack{x\in \zz^d,\\ x\textup{ even}}} w(x) &= \frac{1}{2}\int_{\rr^d} w([x]^0) dx.
\end{align*}
Similarly,
\begin{align*}
\sum_{\substack{x\in \zz^d,\\ x\textup{ odd}}} w(x) &= \frac{1}{2}\int_{\rr^d} w([x]^1) dx,
\end{align*}
whenever the sums are absolutely convergent. 
For $(t,x)\in \rr_{\ge 0}\times \rr^d$, define 
\begin{align*}
p^{(\ve,0)}(t,x) := 
\begin{cases}
\frac{1}{2}\ve^{-d}p([\ve^{-2} t], [\ve^{-1} x]^0) &\text{ if $[\ve^{-2} t]$ is even,}\\
\frac{1}{2}\ve^{-d}p([\ve^{-2} t], [\ve^{-1} x]^1) &\text{ if $[\ve^{-2} t]$ is odd.}
\end{cases} 
\end{align*}
Also, let 
\begin{align*}
h^{(\ve,0)}(t,x) := 
\begin{cases}
\ve^{-2}h_\ve([\ve^{-2} t], [\ve^{-1} x]^0) &\text{ if $[\ve^{-2} t]$ is even,}\\
\ve^{-2} h_\ve([\ve^{-2} t], [\ve^{-1} x]^1) &\text{ if $[\ve^{-2} t]$ is odd.}
\end{cases} 
\end{align*}
\begin{prop}\label{hepprop2}
Take any sequence $(t_n,x_n) \to (t,x)\in \rr_{>0}\times\rr^d$. Then 
\begin{align*}
\lim_{n\to \infty} h^{(\ve_n)}(t_n,x_n) &= \gamma|\nabla f(t,x)|^2. 
\end{align*}
\end{prop}
\begin{proof}
This follows immediately from Lemma \ref{hehe2} and Lemma \ref{heppreplmm}.
\end{proof}
Take any $(t,x)\in \rr_{>0}\times \rr^d$. Define $t_\ve := [\ve^{-2}t]$. If $t_\ve$ is even, let $x_\ve := [\ve^{-1}x]^0$. Otherwise, let $x_\ve:= [\ve^{-1}x]^1$. Then note that $t_\ve$ and $x_\ve$ have the same parity, and $f^{(\ve,0)}(t,x) = f_\ve(t_\ve, x_\ve)$. Thus, by Proposition \ref{fformlmm},
\begin{align*}
&f^{(\ve)}(t,x) = f_\ve(t_\ve,x_\ve) \\
&= \sum_{y\in \zz^d} p(t_\ve,x_\ve-y) g_\ve(y) + \sum_{0\le s\le t_\ve-1} \sum_{y\in \zz^d} p(s, x_\ve-y) h_\ve(t_\ve-s,y).
\end{align*}
Suppose that $t_\ve$ is even. Then $x_\ve$ is also even, and hence $p(s, x_\ve-y) =0$ if  $s$ and $y$ do not have the same parity. Thus, if $s$ is even, then
\begin{align*}
&\sum_{y\in \zz^d} p(s, x_\ve-y) h_\ve(t_\ve-s,y) \\
&= \sum_{\substack{y\in \zz^d,\\ y \text{ even}}} p(s, x_\ve-y) h_\ve(t_\ve-s,y)\\
&= \frac{1}{2}\int_{\rr^d} p(s, x_\ve - [y]^0) h_\ve(t_\ve -s, [y]^0) dy\\
&= \ve^{d+2}\int_{\rr^d} p^{(\ve,0)}(\ve^2 s, \ve(x_\ve-[y]^0))h^{(\ve,0)}(\ve^2(t_\ve-s), \ve [y]^0)dy.
\end{align*}
One gets a similar expression for odd $s$, with $[y]^1$ instead of $[y]^0$. Thus, if we define $\chi(s) := 0$ if $[s]$ is even and $\chi(s):=1$ if $[s]$ is odd, then 
\begin{align*}
 &\sum_{0\le s\le t_\ve-1} \sum_{y\in \zz^d} p(s, x_\ve-y)\\
 &= \ve^{d+2}\int_0^{t_\ve}\int_{\rr^d} p^{(\ve,0)}(\ve^2 [s], \ve(x_\ve-[y]^{\chi(s)}))h^{(\ve,0)}(\ve^2(t_\ve-[s]), \ve [y]^{\chi(s)})dyds\\
 &= \int_0^{\ve^2t_\ve}\int_{\rr^d} p^{(\ve,0)}(\ve^2[\ve^{-2}u], \ve(x_\ve-[\ve^{-1} z]^{\chi(\ve^{-2}u)}))\\
 &\qquad \qquad \qquad \qquad \qquad \cdot h^{(\ve,0)}(\ve^2(t_\ve-[\ve^{-2}u]), \ve [\ve^{-1}z]^{\chi(
\ve^{-2}u)})dzdu.
\end{align*}
Now proceeding as in the proof of Proposition \ref{limitform}, using Lemma \ref{hehe2} and Proposition \ref{hepprop2} instead of Lemma \ref{hehe} and Proposition \ref{hepprop}, and suitable modifications of Lemma \ref{domlmm} and Lemma \ref{remlmm}, we complete the proof of Proposition \ref{limitform} for any subsequential limit of $f^{(\ve,0)}$. The rest of the proof of Theorem \ref{kpzunivthm} now proceeds as before. 

\section*{Acknowledgements}
I thank Ivan Corwin, Persi Diaconis, Chiranjib Mukherjee, Panagiotis Souganidis, and Lexing Ying for helpful comments and references.

\end{document}